\documentclass[12pt,reqno]{amsart}    
\usepackage{amssymb,amsfonts,amsmath,latexsym,eepic,verbatim, bm}
\usepackage{amsthm}
\usepackage{pstricks,pst-node}
\usepackage{tikz-cd}
\usepackage{enumerate}

\psset{unit=1pt}
\psset{arrowsize=4pt 1}
\psset{linewidth=.5pt}

\newgray{lightgray}{.80}

\countdef\x=23
\countdef\y=24
\countdef\z=25
\countdef\t=26

\def\graybox(#1,#2){
\x=#1 \y=#2 
\z=\x \t=\y
\advance\z by 10 
\advance\t by 10 
\psframe[fillstyle=solid,fillcolor=lightgray,linewidth=0pt](\x,\y)(\z,\t) 
\psline[linewidth=.5pt](\x,\y)(\x,\t)(\z,\t)(\z,\y)(\x,\y)}

\def\emptygraybox(#1,#2){
\x=#1 \y=#2 
\z=\x \t=\y
\advance\z by 10 
\advance\t by 10 
\psframe[fillstyle=solid,fillcolor=lightgray,linewidth=0pt,linecolor=lightgray](\x,\y)(\z,\t)}

\def\blankbox(#1,#2){
\x=#1 \y=#2 
\z=\x \t=\y
\advance\z by 10 
\advance\t by 10 
\psframe[linewidth=.5pt](\x,\y)(\z,\t)}

\def\whitebox(#1,#2){
\x=#1 \y=#2 
\z=\x \t=\y
\advance\z by 10 
\advance\t by 10 
\psframe[fillstyle=solid,fillcolor=white,linewidth=0pt](\x,\y)(\z,\t) 
\psline[linewidth=.5pt](\x,\y)(\x,\t)(\z,\t)(\z,\y)(\x,\y)}

\def\whiteboxb(#1,#2){
\x=#1 \y=#2 
\z=\x \t=\y
\advance\z by 10 
\advance\t by 10 
\psframe[fillstyle=solid,fillcolor=white,linewidth=0pt](\x,\y)(\z,\t)}

\newcommand{\define}{\textbf}

\newcommand{\excise}[1]{}

\renewcommand{\P}{\mathbb{P}}

\newcommand{\Z}{\mathbb{Z}}

\newcommand{\isom}{\cong}

\renewcommand{\phi}{\varphi}

\renewcommand{\tilde}{\widetilde}
\renewcommand{\hat}{\widehat}
\renewcommand{\bar}{\overline}

\newcommand{\on}{\operatorname}

\newcommand{\bk}{\mathbf{k}}
\newcommand{\bp}{\mathbf{p}}
\newcommand{\bq}{\mathbf{q}}
\newcommand{\br}{\mathbf{r}}

\DeclareMathOperator{\rk}{rk}

\newcommand{\QQ}{\mathcal{Q}}    
\renewcommand{\O}{\mathcal{O}}   

\newcommand{\triple}{{\bm{\tau}}}
\newcommand{\aaa}{a}

\newcommand{\Eta}{\mathrm{H}}

\newcommand{\Pf}{ \on{Pf} }

\renewcommand{\min}{{\text{min}}}

\newtheorem{thm}{Theorem}
\newtheorem{lem}[thm]{Lemma}

\newtheorem*{thm*}{Theorem}
\newtheorem*{lem*}{Lemma}
\newtheorem*{prop*}{Proposition}
\newtheorem*{cor*}{Corollary}

\theoremstyle{definition}

\newtheorem{rmk}{Remark}

\newtheorem*{defn*}{Definition}
\newtheorem*{rmk*}{Remark}
\newtheorem*{ex*}{Example}

\begin{document}

\title{Chern class formulas for classical-type degeneracy loci}
\date{September 18, 2019}

\author{David Anderson}
\email{anderson.2804@osu.edu}
\address{Department of Mathematics,
The Ohio State University,
Columbus, Ohio, 43210}

\author{William Fulton}
\email{wfulton@umich.edu}
\address{Department of Mathematics,
University of Michigan,
Ann Arbor, Michigan  48109}

\dedicatory{To Piotr Pragacz on the occasion of his sixtieth birthday}
\thanks{DA was partially supported by NSF grant DMS-1502201.}

\begin{abstract}
Employing a simple and direct geometric approach, we prove formulas for a large class of degeneracy loci in types B, C, and D, including those coming from all isotropic Grassmannians.  The results unify and generalize previous Pfaffian and determinantal formulas.  Specializing to the Grassmannian case, we recover the remarkable theta- and eta-polynomials of Buch, Kresch, Tamvakis, and Wilson.  Our method yields streamlined proofs which proceed in parallel for all four classical types, substantially simplifying previous work on the subject.

In an appendix, we develop some foundational algebra and prove several Pfaffian identities.  Another appendix establishes a basic formula for classes in quadric bundles.
\end{abstract}

\maketitle

\setcounter{tocdepth}{1}
\tableofcontents

\section*{Introduction}

A fundamental problem asks for a formula for the cohomology (or Chow) class of a degeneracy locus, as a polynomial in the Chern classes of the vector bundles involved.  In its simplest form, the answer is given by the Giambelli-Thom-Porteous formula: the locus is where two subbundles of a given vector bundle meet in at least a given dimension, and the formula is a determinant.

The aim of this article is to prove formulas for certain degeneracy loci in classical types.  One has maps of vector bundles, or flags of subbundles of a given bundle; degeneracy loci come from imposing conditions on the ranks of maps, or dimensions of intersections.  The particular loci we consider are indexed by {\em triples} of $s$-tuples of integers, $\triple=(\br,\bp,\bq)$ (in type A) or $\triple=(\bk,\bp,\bq)$ (in types B, C, and D).  In type A, each $(r_i,p_i,q_i)$ specifies a rank condition on maps of vector bundles $\rk(E_{p_i} \to F_{q_i}) \leq r_i$; in other types, $E_\bullet$ and $F_\bullet$ are flags of isotropic or coisotropic bundles inside some vector bundle with bilinear form, and each $(k_i,p_i,q_i)$ specifies $\dim(E_{p_i}\cap F_{q_i}) \geq k_i$.

In each case, we will write $\Omega_\triple \subseteq X$ for the degeneracy locus.  Its expected codimension depends on the type.  In fact, to each triple we associate a partition $\lambda(\triple)$ (again  depending on type), whose size is equal to the codimension of $\Omega_\triple$.

The resulting degeneracy loci of type A have a determinantal formula which generalizes that of Giambelli-Thom-Porteous.  The loci corresponding to triples are exactly those defined by {\em vexillary permutations} according to the recipe of \cite{F2}; building on work of Kempf-Laksov, Lascoux-Sch\"utzenberger, and others, it was shown in [{\it op.~cit.}] that
\[
  [\Omega_\triple] = \Delta_{\lambda(\triple)}(c(1),\ldots,c(\ell)) := \det( c(i)_{\lambda_i+j-i} )_{1\leq i,j\leq \ell}.
\]
Here each $c(k)$ is a (total) Chern class $c(F_{q_i}-E_{p_i}) = c(F_{q_i})/c(E_{p_i})$, and $\Delta_\lambda$ is a {\em Schur determinant}; more details will be given in \S\ref{s.typeA}.

In other classical types, work of Pragacz and his collaborators showed that Pfaffians should play a role analogous to determinants in type A, at least for cases where there is a single bundle $E$ and all $F_\bullet$ are isotropic \cite{P1,P2,PR,LP}.  More recent work of Buch, Kresch, and Tamvakis exploits a crucial insight: both determinants and Pfaffians can be defined via {\em raising operators}, and by adopting the raising operator point of view, one can define {\em theta-} and {\em eta-polynomials}, which interpolate between determinants and Pfaffians.  These provide representatives for Schubert classes in non-maximal isotropic Grassmannians; here one has a single isotropic $E$, and a flag of trivial bundles $F_\bullet$, some of which may be coisotropic \cite{BKT1,BKT2,T}.  Wilson extended this idea to define double theta-polynomials, and conjectured that they represent equivariant Schubert classes in non-maximal isotropic Grassmannians \cite{W}.  This was proved in \cite{IM}, and, via a different method, in \cite{TW}.

We will introduce triples $\triple$ for each classical type, and study degeneracy loci defined by $\dim( E_{p_i} \cap F_{q_i} )\geq k_i$, with all $E_\bullet$ isotropic, and $F_\bullet$ either isotropic or coisotropic.  When the $F_\bullet$ are all isotropic, the formulas are (multi-)Pfaffians (as in the preprints \cite{K} and \cite{AF1}); allowing coisotropic conditions presents some subtleties, but leads directly to the definitions of {\it multi-theta polynomials} $\Theta_\lambda$ and {\it multi-eta-polynomials} $\Eta_\lambda$.  Our main theorem is stated in terms of these polynomials.

\begin{thm*}
Let $\triple$ be a triple, and let $\Omega_\triple$ be the corresponding degeneracy locus (of type C, B, or D).

\begin{enumerate}
\item[(C)] In type C, we have
\[
  [\Omega_\triple] = \Theta_{\lambda(\triple)}(c(1),\ldots,c(\ell)).
\]

\item[(B)] In type B, we have
\[
  [\Omega_\triple] = 2^{-k_a}\Theta_{\lambda(\triple)}(c(1),\ldots,c(\ell)).
\]

\item[(D)] In type D, we have
\[
  [\Omega_\triple] = 2^{-k_a}\Eta_{\lambda(\triple)}(c(1),\ldots,c(\ell)).
\]
\end{enumerate}
\end{thm*}

The entries $c(i)$ vary by type, and along with the definitions of $\Theta_\lambda$ and $\Eta_\lambda$, these are specified in Theorems~\ref{t.typeC}, \ref{t.typeB}, and \ref{t.typeD}.  For now, let us mention some special cases.  When the triple has all $q_i\geq0$, the loci are defined by conditions on isotropic bundles, and each formula is a Pfaffian.  
If the triple has all $p_i=p$, the loci come from Grassmannians; in the type C case, we recover the formulas of \cite{IM} and \cite{TW}, and in case the $F$'s are trivial, we recover formulas of \cite{BKT1,BKT2}.  (The type D formula includes a definition of {\em double eta-polynomial}, which is new even in the Grassmannian case.\footnote{Tamvakis recently announced that he has also found such polynomials \cite{T3}.})

The structure of the argument in each type is essentially the same.  First, one has a basic formula for the case where the only condition is $E_{p_1}\subseteq F_{q_1}$, and furthermore $E_{p_1}$ is a line bundle.  Next, there is the case where $E_{p_i}$ has rank $i$, and the conditions are $E_{p_i}\subseteq F_{q_i}$; the formula here is easily seen to be a product, and one uses some elementary algebra to convert the product into a raising operator formula.  (In type A, these loci correspond to {\em dominant permutations}.)  The ``main case'' is where the conditions are $\dim(E_{p_i}\cap F_{q_i})\geq i$; such loci are resolved by a birational map from a dominant locus, and the pushforward can be decomposed into a series of projective bundles.  Finally, a little more elementary algebra reduces the general case to the main case.

Proving the theorem requires only a few general facts.  Some of these are treated in appendices, but we collect four basic formulas here for reference.  Let $E$ be a vector bundle of rank $e$ on a variety $X$.

\begin{enumerate}[(a)]

\item If $L$ is a line bundle on $X$, then
\[
  c_e(E-L) = c_e(E\otimes L^*).
\] \label{ident1}

\item If $L$ is a line bundle on $X$, for any $b\geq 0$ and $j\geq e$ we have
\[
  (-c_1(L))^b\, c_j(E-L) = c_{b+j}(E-L).
\] \label{ident2}

\item If $F'$ is a subbundle of a vector bundle $F$, then
\[
  c(E-F/F') = c(E-F)\,c(F').
\] \label{ident3}

\item  Let $\pi\colon \P(E) \to X$ be the projective bundle, and $Q=\pi^*E/\O(-1)$ the universal quotient bundle.  For any $\sigma\in A^*X$, we have
\[
  \pi_*(\pi^*\sigma \cdot c_j(Q)) = \begin{cases} \sigma & \text{if } j=e-1; \\ 0 &\text{otherwise.} \end{cases}
\] \label{ident4}
\end{enumerate}
(Identities \eqref{ident1}--\eqref{ident3} are easy to deduce from the Whitney sum formula, and \eqref{ident4} follows from the formula for $A^*\P(E)$ as an algebra over $A^*X$.)

We conclude this introduction with some remarks on the development and context of our results.  The {\em double Schubert polynomials} of Lascoux and Sch\"utzenberger, which represent type A degeneracy loci, have many wonderful combinatorial properties.  A problem that received a great deal of attention in the 1990's was to find similar polynomials representing loci of other classical types.  First steps in this direction were taken by Billey and Haiman, who defined (single) Schubert polynomials for types B, C, and D \cite{BH}; double versions were obtained by Ikeda, Mihalcea, and Naruse \cite{IMN} and studied further by Tamvakis \cite{T}.  A simplified development of these double Schubert polynomials follows from the degeneracy locus formulas proved here.  We should point out that in types B, C, and D, any theory of Schubert polynomials involves working not in a polynomial ring, but in a ring with relations; modulo these relations, however, stable formulas are essentially unique.  See \cite{AF1}, or the survey \cite{T2}, for more perspective on this history.

The Schubert varieties and degeneracy loci in types B, C, and D are indexed by signed permutations, and it is natural to ask whether certain signed permutations correspond to Pfaffians, by analogy with the determinantal formulas for vexillary permutations in type A.  In the preprint \cite{AF1}, we identified such a class of {\em vexillary signed permutations},\footnote{Although not obvious from the definitions, the vexillary signed permutations of \cite{AF1} correspond to certain {\em vexillary elements} of the hyperoctahedral group, defined and studied by Billey and Lam \cite{BL}.  It should be interesting to study the permutations arising from the more general triples to be considered here; see the remarks at the end of \S\ref{s.typeC}.} 
defined via triples $\triple$ such that all $q_i\geq 0$.  Following ideas of Kazarian \cite{K}, we proved Pfaffian formulas for vexillary loci, and also studied the relationship between these Pfaffians and the double Schubert polynomials of \cite{IMN}.  We plan to revisit the combinatorics and algebra of Schubert polynomials in separate work.

In the present work, we focus on the Pfaffian formulas and their generalizations.  The setup is heavily influenced by Kazarian's preprint \cite{K}.  We wish to emphasize that our key contribution is the argument itself: by including the more general vexillary loci, we separate algebra (showing a product equals a determinant or Pfaffian, in the ``dominant'' case of the proof) from geometry (constructing a resolution of singularities and pushing forward the formula, in the ``main'' case).

In revisiting our earlier approach, we found the situation was clarified by making explicit the role of raising operators.  For this, we owe a great debt to the work of Buch, Kresch, and Tamvakis, whose remarkable theta- and eta-polynomial formulas have convincingly demonstrated the utility of raising operators in geometry.  This inspired us to apply our geometric method to more general loci, simultaneously generalizing their formulas and yielding shorter and more uniform proofs.

Finally, although the recent prominence of raising operators is due to Buch, Kresch, and Tamvakis, it was Pragacz who first brought them to geometry.  We take both combinatorial and geometric inspiration from his work, and dedicate this article to him on the occasion of his sixtieth birthday.

\medskip
\noindent
{\it Acknowledgements.} We must acknowledge the debt we owe to work of Kazarian \cite{K}, whose method guided our approach to these degeneracy locus problems.  We thank Harry Tamvakis for drawing attention to inaccuracies in the published version of the manuscript, and Jordan Lambert Silva for suggesting a clarification in the definition of a type C triple.  Finally, we thank the referees for detailed comments and suggestions for improving the paper.

\section{Type A revisited} \label{s.typeA}

The determinantal formula describing degeneracy loci in Grassmann bundles --- or more generally, vexillary loci in flag bundles --- is, by now, quite well known; see \cite{KL,F2} for recent versions.  However, our reformulation of its setup and proof will provide a model for the (new) formulas in other types, so we will go through it in detail.

A \define{triple of type A} is the data $\triple = (\br,\bp,\bq)$, where each of $\br$, $\bp$, and $\bq$ is an $s$-tuple of non-negative integers, with
\begin{align*}
 0<&p_1\leq p_2 \leq \cdots \leq p_s \qquad \text{and} \\
 &q_1 \geq q_2 \geq \cdots \geq q_s >0.
\end{align*}
Setting $k_i = p_i - r_i$ and $l_i = q_i - r_i$, we further require that
\begin{align*}
 0<& k_1<k_2< \cdots < k_s \qquad \text{and} \\
   & l_1 \geq l_2 \geq \cdots \geq l_s \geq 0.
\end{align*}
The last condition is equivalent to requiring
\renewcommand{\theequation}{a}
\begin{equation}\label{e.tripleA1}
  k_i - k_{i-1} \leq (q_i-q_{i-1}) - (p_i-p_{i-1})
\end{equation}
\renewcommand{\theequation}{\arabic{equation}}
\setcounter{equation}{0}
for all $i$.

Associated to a triple there is a partition $\lambda=\lambda(\triple)$, defined by setting $\lambda_{k_i} = l_i$, and filling in the remaining parts minimally so that $\lambda_1 \geq \lambda_2 \geq \cdots \geq \lambda_{k_s}\geq 0$.  (An essential triple specifies only the ``corners'' of the Young diagram for $\lambda$, so in a sense it is a minimal way of packaging this information.)

Given a partition $\lambda = (\lambda_1\geq \cdots \geq \lambda_\ell\geq 0)$ and symbols $c(1),\ldots,c(\ell)$, the associated \define{Schur determinant} is
\[
  \Delta_\lambda(c(1),\ldots,c(\ell)) := \det( c(i)_{\lambda_i+j-i} )_{1\leq i,j\leq \ell}.
\]
For a positive integer $\ell$, let $R^{(\ell)}$ be the raising operator
\begin{equation}
  R^{(\ell)} = \prod_{1\leq i<j\leq \ell} (1-R_{ij}),
\end{equation}
where $R_{ij}=T_i/T_j$ is the operator defined in Appendix A.2.  
A simple application of the Vandermonde identity shows that
\begin{equation}
  \Delta_\lambda(c(1),\ldots,c(\ell)) = R^{(\ell)}\cdot (c(1)_{\lambda_1}\cdots c(r)_{\lambda_\ell}),
\end{equation}
and we will use this observation crucially in proving the degeneracy locus formula.

Here is the geometric setup.  On a variety $X$ we have a sequence of vector bundles
\[
  E_{p_1} \hookrightarrow E_{p_2} \hookrightarrow \cdots \hookrightarrow E_{p_s} \xrightarrow{\phi} F_{q_1} \twoheadrightarrow F_{q_2} \twoheadrightarrow \cdots \twoheadrightarrow F_{q_s},
\]
where subscripts indicate ranks; for each $i$, there is an induced map $E_{p_i} \to F_{q_i}$.  The degeneracy locus corresponding to the triple $\triple$ is
\[
  \Omega_\triple := \{ x\in X\, |\, \rk(E_{p_i} \to F_{q_i}) \leq r_i \text{ for } 1\leq i\leq s \}.
\]
This comes equipped with a natural subscheme structure defined locally by the vanishing of certain determinants.

Let $c(k_i) = c(F_{q_i} - E_{p_i})$, and set $c(k) = c(k_i)$ whenever $k_{i-1}<k\leq k_i$.  We also set $\ell=k_s$, and by convention, we always take $k_0=0$.)  With this notation, the degeneracy locus formula can be stated as follows.

\begin{thm}[{cf.~\cite{KL,F2}}]\label{t.typeA}
We have $[\Omega_\triple] = \Delta_{\lambda(\triple)}(c(1),\ldots,c(\ell))$.
\end{thm}

The case where there is only one $F$, so $q_1=\cdots =q_s$, was proved by Kempf and Laksov, starting the modern search for such formulas.  The general case was proved in \cite{F2}.

These formulas are to be interpreted as usual: when the bundles and the map $\phi$ are sufficiently generic, then $\Omega_\triple$ has codimension equal to $|\lambda|$ and the formula is an identity relating the fundamental class of $\Omega_\triple$ with the Chern classes of $E$ and $F$.  In general, the left-hand side should be regarded as a refined class of codimension $|\lambda|$, supported on $\Omega_\triple$; see \cite{F1}.

The proof proceeds in four steps.

\subsection{Basic case}

Assume $s=1$, $p_1=1$, $r_1=0$, so $\triple = (0,1,q_1)$.  The locus is where $E_1 \to F_{q_1}$ vanishes, so it is the zeroes of a section of $E_1^*\otimes F_{q_1}$, a vector bundle of rank $q_1$.  Therefore
\[
  [\Omega_\triple] = c_{q_1}(E_1^*\otimes F_{q_1}) = c_{q_1}(F_{q_1}-E_1),
\]
using Identity~\eqref{ident1} to obtain the second equality.

\subsection{Dominant case}

Assume $p_i=i$ and $r_i=0$ for $1\leq i\leq s$, so $k_i=i$ and $l_i=q_i$.  By imposing one condition at a time, we obtain a sequence
\[
  X = Z_0 \supseteq Z_1 \supseteq Z_2 \supseteq \cdots \supseteq Z_s = \Omega_\triple;
\]
$Z_1$ is the locus where $E_1 \to F_{q_1}$ is zero; $Z_2$ is where also $E_2/E_1 \to F_{q_2}$ is zero; and generally $Z_j$ is defined on $Z_{j-1}$ by the condition that $E_j/E_{j-1} \to F_{q_j}$ be zero.  This is an instance of the basic case, and using the projection formula, it follows that
\begin{equation}\label{e.aDom}
 [\Omega_\triple] = \prod_{j=1}^s c_{q_j}(F_{q_j} - E_j/E_{j-1}).
\end{equation}
Writing $c(j) = c(F_{q_j}-E_j)$ and $t_i = -c_1(E_i/E_{i-1})$, an application of Identity~\eqref{ident3} transforms Equation~\eqref{e.aDom} into
\begin{equation} \label{e.aDom2}
 [\Omega_\triple] = \prod_{j=1}^s \left[ c(j)\cdot \prod_{i=1}^{j-1}(1-t_i) \right]_{q_j}.
\end{equation}
Using Identity~\eqref{ident2}, this becomes
\begin{align}\label{e.aDom3}
  [\Omega_\triple] &= \left( \prod_{1\leq i<j\leq s} (1-R_{ij}) \right) \cdot (c(1)_{q_1} c(2)_{q_2} \cdots c(s)_{q_s} ) \\
   &= R^{(s)}\cdot (c(1)_{q_1} c(2)_{q_2} \cdots c(s)_{q_s} ). \nonumber
\end{align}
In other words, the product \eqref{e.aDom} is equal to the determinant $\Delta_\lambda(c(1),\ldots,c(s))$, where $\lambda = (q_1 \geq q_2 \geq \cdots \geq q_s)$. 
(To deduce \eqref{e.aDom3} from \eqref{e.aDom2}, use Identity~\eqref{ident2} and descending induction on $k$ to show \eqref{e.aDom2} equals
\begin{align*}
 & \left(\prod_{\substack{ 1\leq i<j\leq s \\ k<j}} (1-R_{ij})\right) \cdot c'(1)_{q_1}\cdots c'(k)_{q_{k}} \cdot c(k+1)_{q_{k+1}} \cdots c(s)_{q_s}, 
\end{align*}
where $c'(k)=c(F_{q_k}-E_k/E_{k-1}) = c(k)\cdot \prod_{i=1}^{k-1}(1-t_i)$.  The case $k=s$ is \eqref{e.aDom2}, and the case $k=1$ is \eqref{e.aDom3}.)

\subsection{Main case}

Assume $k_i=p_i-r_i=i$ for $1\leq i\leq s$.  There is a sequence of projective bundles
\[
  X = X_0 \leftarrow X_1 = \P(E_{p_1}) \leftarrow X_2 = \P(E_{p_2}/D_1) \leftarrow \cdots \leftarrow X_s = \P(E_{p_s}/D_{s-1}),
\]
where, suppressing notation for pullbacks of bundles, $D_j/D_{j-1} \subseteq E_{p_j}/D_{j-1}$ is the tautological line bundle on $X_j$.  Let us write $\pi^{(j)}\colon X_j \to X_{j-1}$ for the projection, and $\pi\colon X_s \to X$ for the composition of all the $\pi^{(j)}$'s.

On $X_s$, there is the locus $\tilde{\Omega}$ where $D_i/D_{i-1} \to F_{q_i}$ is zero for $1\leq i\leq s$.  This is an instance of the dominant case, so in $A^*X_s$ we have
\begin{equation}\label{e.aMain}
  [\tilde\Omega] = R^{(s)} \cdot \tilde{c}(1)_{\tilde{\lambda}_1}\cdots\tilde{c}(s)_{\tilde{\lambda}_s},
\end{equation}
where $\tilde\lambda_j = q_j$ and $\tilde{c}(j) = c(F_{q_j}-D_j)$.  Furthermore, $\pi$ maps $\tilde\Omega$ birationally onto $\Omega_\triple$; it is an isomorphism over the dense open set where $\rk(E_{p_i} \to F_{q_i}) = r_i$.  (Take $D_i$ to be the kernel.)  So $[\Omega_\triple] = \pi_*[\tilde\Omega]$.

To compute this pushforward, use Identity~\eqref{ident3} to write $\tilde{c}(j) = c(j)\cdot c(E_{p_j}/D_j)$, recalling that $c(j) = c(F_{q_j}-E_{p_j})$.  Note that $E_{p_j}/D_j$ is the tautological quotient bundle for $\pi^{(j)}\colon X_j \to X_{j-1}$.  By Identity~\eqref{ident4}, we have
\[
  \pi^{(j)}_*\left( c(j)_a\cdot c_b(E_{p_j}/D_j) \right) = c(j)_a
\]
when $b=p_j-j$, and this pushforward equals $0$ otherwise.  Therefore
\[
  \pi^{(j)}_*(\tilde{c}(j)_k) = \pi^{(j)}_*\left( \sum_{a+b=k} c(j)_a \cdot c_b(E_{p_j}/D_j) \right) = c(j)_{k-p_j+j}.
\]
Applying $\pi_*$ to \eqref{e.aMain} and using linearity of the raising operator yields
\begin{align*}
  [\Omega_\triple] &= \pi_*\left( R^{(s)}\cdot \tilde{c}(1)_{\tilde{\lambda}_1}\cdots\tilde{c}(s)_{\tilde{\lambda}_s}\right) \\
                   &= R^{(s)} \pi_*\left(  \tilde{c}(1)_{\tilde{\lambda}_1}\cdots\tilde{c}(s)_{\tilde{\lambda}_s}\right) \\
                   &= R^{(s)} c(1)_{{\lambda}_1}\cdots {c}(s)_{{\lambda}_s} \\
                   &= \Delta_\lambda(c(1),\ldots,c(s)).
\end{align*}

\subsection{General case}

Any triple $\triple=(\br,\bp,\bq)$ can be ``inflated'' to $\triple'=(\br',\bp',\bq')$ with $k'_i=p'_i-r'_i=i$, without essentially altering the locus $\Omega_\triple$ or the polynomial representing it.  Suppose $k_i-k_{i-1}>1$, so either $q_{i-1}>q_{i}$ or $p_{i-1}<p_i$ (or both).  If $q_{i-1}>q_i$, then inserting $(r_{i}+1,p_{i},q_{i}+1)$ between the $(i-1)^{\text{st}}$ and $i^{\text{th}}$ positions produces a new triple $\triple'$ with $\lambda(\triple)=\lambda(\triple')$.  If there is a bundle $F_{q_i+1}$ fitting into $F_{q_{i-1}}\twoheadrightarrow F_{q_i+1} \twoheadrightarrow F_{q_{i}}$, then one easily checks $\Omega_{\triple'} = \Omega_\triple$.  In general, it can be arranged for such an $F_{q_i+1}$ to exist by passing to an appropriate projective bundle; then the locus $\Omega_{\triple'}$ maps birationally to the original $\Omega_\triple$.  (In the case $q_{i-1}=q_i$, then $p_{i-1}<p_i-1$ and one proceeds similarly, by inserting a bundle between $E_{p_{i-1}}$ and $E_{p_i}$.)  The fact that $\Delta_{\lambda}(c'(1),\ldots,c'(\ell)) = \Delta_\lambda(c(1),\ldots,c(\ell))$ is a special case of \S\ref{ss.A.theta}, Lemma~\ref{l.inflate}.

This concludes the proof.
\qed

\bigskip

\section{Type C: symplectic bundles} \label{s.typeC}

A \define{triple of type C} is $\triple = (\bk,\bp,\bq)$, with
\begin{align*}
 0<&k_1<k_2<\cdots<k_s , \\
 &p_1\geq p_2 \geq \cdots \geq p_s >0, \\
 &q_1 \geq q_2 \geq \cdots \geq q_s.
\end{align*}
The $q_i$ are allowed to be negative, but not zero, and if $p_s=1$ then all $q_i$ must be positive.  Since the difference between positive and negative $q$ plays a major role, let $\aaa=\aaa(\triple)$ be the index such that $q_{\aaa}>0>q_{\aaa+1}$ (allowing $\aaa=0$ and $\aaa=s$ for the cases where all $q$'s are negative or all $q$'s are positive, respectively).  We will also require five further conditions, listed as \eqref{e.tripleC1}--\eqref{e.tripleC5} below, which arise naturally from the geometric setup.  
Consider an even-rank vector bundle $V$, equipped with a symplectic form and two flags of subbundles
\begin{align*}
  E_{p_1} \subset E_{p_2} \subset \cdots \subset E_{p_s} \subset V, \\
  F_{q_1} \subset F_{q_2} \subset \cdots \subset F_{q_s} \subset V.
\end{align*}
When $q>0$, the subbundles $F_q$ are isotropic; when $q<0$, $F_q$ is coisotropic; and all the bundles $E_p$ are isotropic.  If the rank of $V$ is $2n$, the isotropic bundles $E_p$ and $F_q$ (for $q>0$) have rank $n+1-p$ and $n+1-q$, respectively; and for $q<0$, the coisotropic bundles $F_q$ have rank $n-q$.  (So the order on the $p$'s and $q$'s is compatible with the inclusion of bundles of corresponding ranks.)

The degeneracy locus is
\[
  \Omega_\triple = \{ x\in X \,|\, \dim(E_{p_i} \cap F_{q_i}) \geq k_i \text{ for } 1\leq i\leq s \}.
\]
Note that $E_1^\perp=E_1$, so a condition on its intersection with a coisotropic space is equivalent to one for the intersection with an isotropic space; we shall prefer the latter, and this explains why negative $q$'s are prohibited when $p=1$.  Demanding that the rank conditions be feasible, and generically attained with equality, leads to the further requirements on the triple $\triple$.

The conditions on $\triple$ are likely to appear somewhat technical at first.  The reader may find it helpful to first assume all $q$'s are positive, which is the simplest case.  The next easiest case is when $k_i=i$ for all $i$, which corresponds to the ``main case'' in the proof.  Linear-algebraic reasons for the conditions, as well as combinatorial explanations and the relationship with \cite{BKT1}, can be found in the remarks at the end of this section.

In what follows, when indices fall out of the range $[1,s]$, we use conventions so that the inequalities become trivial, e.g., $k_0=0$, $q_0=+\infty$, and $q_{s+1}=-\infty$.

\newcounter{tempc}
\setcounter{tempc}{\value{equation}}
\setcounter{equation}{0}
\renewcommand{\theequation}{c\arabic{equation}}

First, for $i\leq a$, we require
\begin{align}
  k_{i}-k_{i-1} &\leq (p_{i-1}-p_{i}) + (q_{i-1}-q_{i}).
  \label{e.tripleC1}
\end{align}
When $k_i=i$ for all $i$, this says that either $p_{i-1}>p_i$ or $q_{i-1}>q_i$.  
When all $q$'s are positive, this is the only condition required of a triple $\triple$.

The other requirements on $\triple$ concern negative $q$'s, so they describe intersections of an isotropic $E_p$ with a coisotropic $F_q$.

For each $j\leq \aaa$, let $m(j) =  \min\{ m \,|\, q_j+(k_j-k_{j-1}-1)\geq q_{m}\}$.  The second condition is this:
\begin{align}\label{e.prohibitC2}
& \text{The negative values } \nonumber \\
&    \qquad    {-q}_j,\, -q_j-1,\, \ldots,\, -q_j-(k_j-k_{m(j)-1}-1) \\
&       \text{are all prohibited as values of }q_i\text{ for }i> \aaa. \nonumber
\end{align}
(Here is an equivalent, algorithmic condition: let $N_0 = \{1,2,3,\ldots\}$; form $N_1$ by removing the $k_1$ consecutive elements of $N_0$ starting at $q_1$; then form $N_2$ by removing $k_2-k_1$ consecutive elements of $N_1$, starting at the $q_2$-th element; and so on up to $N_\aaa$.  The absolute values of $q_i$ for $i\geq a$ are required to lie in $N_\aaa$.)  When $k_i=i$ for all $i$, this simply says that a negative value $q_i$ cannot appear if $|q_i|$ has already appeared as some $q_j$ for $j<i$.

Next, for each $i>\aaa$, so $q_i<0$, we define
\[
  \rho_{k_i} := k_j,
\]
where $j$ is the index such that $q_j>-q_i>q_{j+1}$, and require that
\begin{equation}\label{e.tripleC3}
  k_i - \rho_{k_i} \leq -q_i.
\end{equation}
When $k_i=i$ for all $i$, $\rho$ has an easy characterization as $\rho_i=\#\{j<i\,|\, q_j>-q_i\}$.

The fourth condition is that for $i>\aaa+1$,
\begin{equation}\label{e.tripleC4}
 (k_i-k_{i-1}) + (\rho_{k_{i-1}}-\rho_{k_i}) \leq (p_{i-1}-p_i) + (q_{i-1}-q_i),
\end{equation}
and if equality holds, then either $\rho_{k_{i-1}}=\rho_{k_i}$ or $k_i=k_{i-1}+1$.  This refines condition \eqref{e.tripleC1}.

The fifth and final condition is that
\begin{equation}\label{e.tripleC5}
  k_i \geq 1-p_i-q_i+\rho_{k_i}
\end{equation}
for $i>\aaa$.  (In conjunction with \eqref{e.tripleC4}, it suffices to require \eqref{e.tripleC5} only for $i=s$.)  

\renewcommand{\theequation}{\arabic{equation}}
\setcounter{equation}{\value{tempc}}

As in type A, a triple of type C has an associated partition $\lambda=\lambda(\triple)$, defined by
\begin{align*}
 \lambda_{k_i} = \begin{cases} p_i+q_i-1 & \text{if } i\leq \aaa; \\
                               p_i+q_i+k_i-1-\rho_{k_i} & \text{if } i> \aaa. \end{cases}
\end{align*}
The conditions on $\triple$ imply that
\[
  \lambda_{k_1} > \lambda_{k_2} > \cdots > \lambda_{k_\aaa} > \lambda_{k_{\aaa+1}} \geq \cdots \geq \lambda_{k_s} \geq 0.
\]
The other parts of $\lambda$ are defined by filling in $\lambda_k$ minimally subject to these inequalities (strict if $k < k_\aaa$, weak if $k>k_\aaa$).

Generally, given a unimodal sequence of nonnegative integers $\rho=(\rho_1,\ldots
,\rho_\ell)$, we will say that a partition $\lambda = (\lambda_1 \geq \cdots \geq \lambda_\ell)$ is \define{$\rho$-strict} if the sequence $\mu_j=\lambda_j+\rho_j$ is nonincreasing.

From this point of view, it is useful to extend the definition of $\rho$ given above to a unimodal sequence of integers $\rho(\triple) = (\rho_1,\rho_2,\ldots,\rho_{k_s})$, as follows.  As before, for $i>\aaa$, define $\rho_{k_i} = k_j$, where $q_j>-q_i>q_{j+1}$.  For $k\leq k_{\aaa}$, set $\rho_k=k-1$.  Then fill in the remaining entries by setting $\rho_k = \rho_{k_i}$ for $i> a$ and $k_{i-1}<k\leq k_i$.  This means $\rho_{k_{\aaa}+1} \geq \cdots \geq \rho_{k_s}$.

The conditions \eqref{e.tripleC1}--\eqref{e.tripleC5} are equivalent to requiring:
\begin{itemize}
\item $\lambda(\triple)$ is a $\rho(\triple)$-strict partition, with $\lambda_{k_{\aaa}}>\lambda_{k_{\aaa}+1}$;

\item if $\lambda_{j}=\lambda_{j+1}+1$, then $\rho_i\neq j$ whenever $i>j+1$; and

\item if $\lambda_{k_j}=\lambda_{k_{j+1}}$, then either $\rho_{k_j}=\rho_{k_{j+1}}$ or $k_{j+1}=k_j+1$.
\end{itemize}
(See Remark~\ref{r.kstrict} for the relation with the corresponding notion from \cite{BKT1}.)

For example, suppose $\triple = (\, 1\,3\,5\,6\,7\,9\, ,\; 9\,7\,6\,5\,2\,2\, ,\; 6\,3\,\bar{2}\,\bar{5}\,\bar{7}\,\bar{9}\, )$, using a bar to indicate negative integers.  Here $\aaa=2$, $\rho=(0,1,2,3,3,1,0,0,0)$, and $\lambda(\triple) = (14, 10, 9, 5, 5, 4, 1, 1, 1)$.  (See Figure~\ref{f.shape}.)

\begin{figure}[h]

\begin{center}
\pspicture(100,30)(0,-60)

\psset{unit=1.00pt}

\whitebox(-30,20)
\whitebox(-20,20)
\whitebox(-10,20)
\whitebox(0,20)
\whitebox(10,20)
\whitebox(20,20)
\whitebox(30,20)
\whitebox(40,20)
\whitebox(50,20)
\whitebox(60,20)
\whitebox(70,20)
\whitebox(80,20)
\whitebox(90,20)
\whitebox(100,20)

\graybox(-30,10)
\whitebox(-20,10)
\whitebox(-10,10)
\whitebox(0,10)
\whitebox(10,10)
\whitebox(20,10)
\whitebox(30,10)
\whitebox(40,10)
\whitebox(50,10)
\whitebox(60,10)
\whitebox(70,10)

\graybox(-30,0)
\graybox(-20,0)
\whitebox(-10,0)
\whitebox(0,0)
\whitebox(10,0)
\whitebox(20,0)
\whitebox(30,0)
\whitebox(40,0)
\whitebox(50,0)
\whitebox(60,0)
\whitebox(70,0)

\graybox(-30,-10)
\graybox(-20,-10)
\graybox(-10,-10)
\whitebox(0,-10)
\whitebox(10,-10)
\whitebox(20,-10)
\whitebox(30,-10)
\whitebox(40,-10)

\graybox(-30,-20)
\graybox(-20,-20)
\graybox(-10,-20)
\whitebox(0,-20)
\whitebox(10,-20)
\whitebox(20,-20)
\whitebox(30,-20)
\whitebox(40,-20)

\graybox(-30,-30)
\whitebox(-20,-30)
\whitebox(-10,-30)
\whitebox(0,-30)
\whitebox(10,-30)

\whitebox(-30,-40)

\whitebox(-30,-50)

\whitebox(-30,-60)

\endpspicture
\end{center}

\caption{The shape for a $\rho(\triple)$-strict partition $\lambda(\triple)$.  The boxes of $\rho$ are shaded. \label{f.shape}}

\end{figure}

Given an integer $\ell>0$, a sequence of nonnegative integers $\rho = (\rho_1,\ldots,\rho_\ell)$ with $\rho_j<j$, define the raising operator
\begin{align*}
  R^{(\rho,\ell)} &=  \left(\prod_{1\leq i<j\leq \ell} (1-R_{ij}) \right)\left(\prod_{1\leq i\leq \rho_j < j \leq \ell} (1+R_{ij})^{-1}\right).
\end{align*}
Inspired by \cite{BKT1}, given symbols $c(1),\ldots,c(\ell)$, and a $\rho$-strict partition $\lambda$, we define the \define{theta-polynomial} to be
\[
  \Theta^{(\rho)}_\lambda( c(1),\ldots,c(\ell) ) = R^{(\rho,\ell)} \cdot (c(1)_{\lambda_1}\cdots c(\ell)_{\lambda_{\ell}}).
\]
When $\rho=\emptyset$, $\Theta^{(\rho)}_\lambda$ is a Schur determinant, and when $\rho_j = j-1$, $\Theta^{(\rho)}_\lambda$ is a Schur Pfaffian.  

For a triple and the corresponding geometry described above, let $c(k_i) = c(V-E_{p_i}-F_{q_i})$, and for general $k$, take $c(k)=c(k_i)$ where $i$ is minimal such that $k_i\geq k$.  Set $\ell=k_s$.

\begin{thm}\label{t.typeC}
We have $[\Omega_\triple] = \Theta^{(\rho(\triple))}_{\lambda(\triple)}(c(1),c(2),\ldots,c(\ell))$.
\end{thm}

The proof follows the same pattern as the one we saw in type A.  As before, there are four cases.  The appearance of Pfaffians in the formulas can be traced to a basic fact about isotropic subbundles: if $D\subset V$ is isotropic, then the symplectic form identifies $D^\perp$ with $(V/D)^*$.  This is used in the second case.

\subsection{Basic case}

Take $s=1$, $k_1=\ell=1$, and $p_1=n$, so $E_n$ is a line bundle and we are looking at $\Omega_\triple = \{ x\,|\, E_n \subseteq F_{q_1} \}$.  Equivalently, $E_n \to V/F_{q_1}$ is zero, so Identity~\eqref{ident1} lets us write
\[
  [\Omega_\triple] = c_{\lambda_1}(V/F_{q_1} \otimes E_n^*) = c_{\lambda_1}(V-F_{q_1}-E_n)
\]
where
\[
  \lambda_1 = \rk( V/F_{q_1} ) = \begin{cases} n+q_1-1 &\text{ if } q_1>0; \\ n+q_1 &\text{ if } q_1<0. \end{cases}
\]

\subsection{Dominant case}

Now take $k_i=i$ and $p_i = n+1-i$, for $1\leq i\leq s$.  Write $D_i = E_{p_i}$, so this is a vector bundle of rank $i$, and we have $D_1 \subset D_2 \subset \cdots \subset D_s \subset V$.  Letting $Z_j$ be the locus in $X$ where $D_i \subseteq F_{q_i}$ for all $i\leq j$, we have
\[
 X \supseteq Z_{1} \supseteq \cdots \supseteq Z_{s-1} \supseteq Z_s = \Omega_{\triple}.
\]
On $Z_{j-1}$, we have $D_{j-1} \subseteq F_{q_{j-1}} \subseteq F_{q_j} \subseteq V$.  Since $D_j$ is isotropic, we automatically have $D_j \subseteq D_{j-1}^\perp$, so $Z_j$ is defined by the condition
\[
  D_j/D_{j-1} \subseteq (F_{q_{j}}\cap D_{j-1}^\perp)/D_{j-1} \subseteq D_{j-1}^\perp/D_{j-1},
\]
or equivalently, $D_j/D_{j-1} \to D_{j-1}^\perp/ (F_{q_{j}}\cap D_{j-1}^\perp)$ is zero.

When $j\leq \aaa$, the bundle $F_{q_j}$ is isotropic, and this implies $F_{q_j} \subseteq D_{j-1}^\perp$.  In this case, $Z_j$ is equivalently defined by $D_j/D_{j-1} \subseteq F_{q_j}/D_{j-1} \subseteq D_{j-1}^\perp/D_{j-1}$, 
and the basic case says 
\begin{align*}
  [Z_j] &= [Z_{j-1}] \cdot c_{\lambda_j}(D_{j-1}^\perp/D_{j-1} - F_{q_j}/D_{j-1} - D_j/D_{j-1}),
\end{align*}
with $\lambda_j = q_j + p_j - 1 = q_j + n - j$.  (That is, $\lambda_j=\rk(D_{j-1}^\perp/F_{q_j})$.)

When $j>\aaa$ (so $q_j<0$), we have
\[
  D_{j-1}^\perp/(F_{q_j}\cap D_{j-1}^\perp) = (D_{j-1}\cap F_{q_j}^\perp)^\perp/F_{q_j} = D_{\rho_j}^\perp/F_{q_j}
\]
by property \eqref{e.geo-rho} from Remark~\ref{r.linalg}.  (Recall that $\rho_j=i$, where $q_i>-q_j>q_{i+1}$.)  The basic case says
\begin{align*}
  [Z_j] &= [Z_{j-1}] \cdot c_{\lambda_j}(D_{\rho_j}^\perp/D_{j-1} - F_{q_j}/D_{j-1} - D_j/D_{j-1}),
\end{align*}
with $\lambda_j = n+q_j-i = p_j+q_j+j-1-\rho_j$.  (That is, $\lambda_j=\rk(D_{\rho_j}^\perp/F_{q_j})$.)

With $\rho_j=j-1$ for $1\leq j\leq \aaa$, it follows that
\begin{align*}
  [\Omega_\triple]  & = \left( \prod_{j=1}^s c_{\lambda_j}(D_{\rho_j}^\perp/D_{j-1} - F_{q_j}/D_{j-1} - D_j/D_{j-1}) \right).
\end{align*}
Using the symplectic form to identify $D_{\rho_j}^\perp$ with $(V/D_{\rho_j})^*$, we have
\[
  c(D_{\rho_j}^\perp/D_{j-1} - F_{q_j}/D_{j-1} - D_j/D_{j-1}) = c(V-D_j-F_{q_j})\cdot c(D_{j-1}-D_{\rho_j}^*);
\]
setting $t_j = -c_1(D_j/D_{j-1})$ and applying Identity~\eqref{ident3}, this becomes
\begin{align}\label{e.typeCt}
  [\Omega_\triple] &=  \prod_{j=1}^s \left[c(V-D_j-F_{q_j}) \frac{ \prod_{i=1}^{j-1} (1-t_i)}{\prod_{i=1}^{\rho_j} (1+t_i) }
\right]_{\lambda_j}.
\end{align}
Setting $c(j) = c(V-D_j-F_{q_j})=c(V-E_{p_j}-F_{q_j})$ and applying Identity~\eqref{ident2}, this is
\begin{equation}\label{e.typeCop}
  [\Omega_\triple] = R^{(\rho,s)} c(1)_{\lambda_1} \cdots c(s)_{\lambda_s}.
\end{equation}

\subsection{Main case}

Here we only assume $k_i=i$, for $1\leq i\leq s$.  For $j\leq a$, we set $\rho_j=j-1$, and for $j> a$, we have $\rho_j = i$, where $q_i>-q_j>q_{i+1}$.  We have the same sequence of projective bundles as in type A,
\[
  X = X_0 \leftarrow X_1 = \P(E_{p_1}) \leftarrow X_2 = \P(E_{p_2}/D_1) \leftarrow \cdots \leftarrow X_s = \P(E_{p_s}/D_{s-1});
\]
again, $D_j/D_{j-1} \subseteq E_{p_j}/D_{j-1}$ is the tautological subbundle on $X_j$.  Write $\pi^{(j)} \colon X_j \to X_{j-1}$ for the projection, and $\pi\colon X_s \to X$ for the composition.

On $X_s$, we have the locus $\tilde\Omega = \{ D_i \subseteq F_{q_i}\,|\, \text{ for } 1\leq i\leq s\}$.  By the previous case, as a class in $A^*X_s$ we have
\[
  [\tilde\Omega] = R^{(\rho,s)} \tilde{c}(1)_{\tilde\lambda_1} \cdots \tilde{c}(s)_{\tilde\lambda_s},
\]
where $\tilde{c}(j) = c(V-D_j-F_{q_j})$ and
\[
 \tilde\lambda_j = \begin{cases} q_j+n-j &\text{ if }j\leq a; \\ n+q_j-\rho_j &\text{ if } j> a . \end{cases}
\]
Furthermore, $\pi$ maps $\tilde\Omega$ birationally onto $\Omega_\triple$.

For each $i$, using Identity \eqref{ident4} we have
\[
  \pi^{(i)}_*\tilde{c}(i)_{\tilde\lambda_i} = c(i)_{\lambda_i},
\]
where $\lambda = \lambda(\triple)$.  (Note that $\lambda_i = \tilde\lambda_i - \rk(E_{p_i}/D_{i}) = \tilde\lambda_i - n - 1 + p_i + i$.)

This case follows, since
\begin{align*}
  \pi_*[\tilde\Omega] &= \pi_*R^{(\rho,s)}\tilde{c}(1)_{\tilde\lambda_1} \cdots \tilde{c}(s)_{\tilde\lambda_s} \\
     &= R^{(\rho,s)} \pi_*\tilde{c}(1)_{\tilde\lambda_1} \cdots \tilde{c}(s)_{\tilde\lambda_r} \\
     &= R^{(\rho,s)} {c}(1)_{\lambda_1} \cdots {c}(s)_{\lambda_s} \\
     &= \Theta^{(\rho)}_\lambda(c(1),\ldots,c(s)).
\end{align*}

\subsection{General case}

As in type A,  any triple $\triple=(\bk,\bp,\bq)$ can be inflated to a triple $\triple'=(\bk',\bp',\bq')$ having $k'_i=i$, for $1\leq i\leq \ell=k_s$, so that $\triple$ and $\triple'$ define equivalent degeneracy loci.  To do this, it suffices to insert $(k',p',q')$ between $(k_{i-1},p_{i-1},q_{i-1})$ and $(k_i,p_i,q_i)$ whenever $k_i-k_{i-1}>1$.  If $p_{i-1}>p_i$, one can always insert $(k_i-1,p_i+1,q_i)$.  If $q_{i-1}>q_i$, one can insert $(k_i-1,p_i,q_i+1)$.  (For $i\leq a$, condition \eqref{e.prohibitC2} ensures the result is still a triple; for $i>a$, this is guaranteed by \eqref{e.tripleC4}.)  
Note that $\lambda(\triple)=\lambda(\triple')$ and $\rho(\triple)=\rho(\triple')$.

When all the additional bundles $E_{p'_i}$ and $F_{q'_i}$ are present on $X$, one has $\Omega_{\triple'}=\Omega_{\triple}$; otherwise, they can be found by appropriate projective bundles, producing a birational map $\Omega_{\triple'} \to \Omega_\triple$.

The ``main case'' provides a theta-polynomial formula using the triple $\triple'$ and classes $c'(k'_i) = c'(i)= c(V-E_{p'_i}-F_{q'_i})$.  For $k\leq k_a$ and $m\geq \lambda_k$, we have relations
\begin{align*}
  c(k)_m^2 + 2\sum_{j>0} (-1)^j\, c(k)_{m+j}\, c(k)_{m-j} = 0.
\end{align*}
Indeed, we may suppose $k=k_i$ for some $i$, so the LHS can be written
\begin{align*}
 \sum_{j=-\infty}^\infty (-1)^j c(k)_{m+j}\, c(k)_{m-j}  &= (-1)^{m} c_{2m}(V-E_{p_i}-F_{q_i} + V^* - E_{p_i}^* - F_{q_i}^* ) \\
  & = (-1)^{m} c_{m}( E_{p_i}^\perp / E_{p_i} + (F_{q_i}^\perp/F_{q_i})^* ),
\end{align*}
which vanishes because the bundle in the last line has rank $2\lambda_k-2$.  
These are the relations required in the hypothesis of \S\ref{ss.A.theta}, Lemma~\ref{l.inflate}, which shows that $\Theta^{(\rho)}_{\lambda}(c'(1),\ldots,c'(\ell)) = \Theta^{(\rho)}_\lambda(c(1),\ldots,c(\ell))$. \qed

\bigskip

As mentioned above, in extreme cases the theta-polynomial is a determinant or Pfaffian.  To include the case where $\ell$ is odd, we recall that Pfaffians can be defined for odd matrices $(m_{ij})$ by introducing a zeroth row, $m_{0j}$ (see Appendix~\ref{ss.A.pf}).

\begin{cor*}
If all $q_i>0$, then
\[
  [\Omega_\triple] = \Pf_\lambda( c(1), \ldots, c(\ell) ),
\]
where the right-hand side is defined to be the Pfaffian of the matrix $(m_{ij})$, with
\[
  m_{ij} = c(i)_{\lambda_i}\, c(j)_{\lambda_j} + 2\sum_{a>0} (-1)^a c(i)_{\lambda_i+a}\, c(j)_{\lambda_j-a},
\]
and when $\ell$ is odd, the matrix is augmented by $m_{0j}=c(j)_{\lambda_j}$ for $0< j\leq \ell$.
\end{cor*}

\noindent
This follows from the Proposition of \S\ref{ss.A.raising}.  Our conventions for Pfaffians, as in the Appendix, are that one forms a skew-symmetric matrix by defining $m_{ji}=-m_{ij}$ for $i<j$, and $m_{ii}=0$.  (In fact, following the proof of \cite[Theorem~1.1]{K}, it is equivalent to  define the $m_{ij}$ for all $i,j$ by the formula of the Corollary, but we do not need this.)

\begin{rmk}\label{r.linalg}
Here are geometric reasons for the conditions on a triple.  Only elementary linear algebra and basic facts about nondegenerate bilinear forms are needed---an isotropic subspace has dimension at most half that of the ambient space; $(E\cap F)^\perp = E^\perp + F^\perp$; and $(V/E)^* \isom E^\perp$.

Condition~\eqref{e.tripleC1} has a simple explanation:
\[
  (E_{p_{i}} \cap F_{q_{i}})/(E_{p_{i-1}} \cap F_{q_{i-1}}) \subseteq E_{p_{i}}/E_{p_{i-1}} \oplus F_{q_{i}}/F_{q_{i-1}}.
\]

Condition \eqref{e.tripleC3} arises from asking that the coisotropic condition imposed by $F_{q_i}$ be independent of any isotropic conditions, in the following sense.  With $i$ and $j$ as in the definition of $\rho$, so $i>a$ and $\rho_{k_i}=k_j$, we have $F_{q_j} \subseteq F_{q_i}^\perp$ (and this is the largest among the $F$'s which is contained in $F_{q_i}^\perp$).  For generic $E_p \supseteq E_{p_j}$, we require
\[
 \dim(E_{p} \cap F_{q_i}^\perp) = \dim(E_{p_j} \cap F_{q_i}^\perp) = \dim(E_{p_j} \cap F_{q_j}) = k_j,
\]
and in particular this holds for $p=p_i$ and $p=p_{i-1}$.  Thus
\renewcommand{\theequation}{$*$}
\begin{equation}\label{e.geo-rho}
  \rho_{k_i} = \dim(E_{p_i} \cap F_{q_i}^\perp) = \dim(E_{p_{i-1}} \cap F_{q_i}^\perp) \\
\end{equation}
for generic spaces, a formulation which is used in proving the main theorem of this section.  
\renewcommand{\theequation}{\arabic{equation}}
\addtocounter{equation}{-1}
Now Condition \eqref{e.tripleC3} is a consequence of  \eqref{e.geo-rho} and the fact that
\[
  (E_{p_i}\cap F_{q_i})/(E_{p_i}\cap F_{q_i}^\perp) = ((E_{p_i}\cap F_{q_i}) + F_{q_i}^\perp) / F_{q_i}^\perp \subseteq F_{q_i} / F_{q_i}^\perp.
\]
(Since $\dim( F_{q_i} / F_{q_i}^\perp ) = -2q_i$, an isotropic subspace has dimension at most $-q_i$.)

Condition \eqref{e.tripleC4} is similar to \eqref{e.tripleC1}: one has
\[
  (E_{p_{i}} \cap F_{q_{i}})/(E_{p_{i-1}} \cap F_{q_{i-1}}) \hookrightarrow E_{p_{i}}/E_{p_{i-1}} \oplus F_{q_{i}}/F_{q_{i-1}} \twoheadrightarrow (E_{p_i}+F_{q_i})/(E_{p_{i-1}}+F_{q_{i-1}}),
\]  
and one sees $\dim( (E_{p_i}+F_{q_i})/(E_{p_{i-1}}+F_{q_{i-1}}) ) \geq \rho_{k_{i-1}}-\rho_{k_i}$ from the following diagram:
\[
\begin{tikzcd}
  (V/(E_{p_i}+F_{q_i}))^*  \ar[r,hook] & (V/(E_{p_{i-1}}+F_{q_{i-1}}))^* \ar[r,two heads]  & ((E_{p_i}+F_{q_i})/(E_{p_i}+F_{q_i}))^* \\
 E_{p_i}\cap F_{q_i}^\perp \ar[r,hook] \ar[u,hook] & E_{p_{i-1}} \cap F_{q_{i-1}}^\perp \ar[u,hook] \ar[r,two heads]  & (E_{p_{i-1}} \cap F_{q_{i-1}}^\perp)/(E_{p_i}\cap F_{q_i}^\perp). \ar[u,hook,swap,"\alpha"]
\end{tikzcd}
\]
The map $\alpha$ is injective because $(E_{p_i}^\perp \cap F_{q_i}^\perp) \cap (E_{p_i}\cap F_{q_i}^\perp) = E_{p_i}\cap F_{q_i}^\perp$.

Condition \eqref{e.tripleC5} comes from writing $E_{p_i} \cap F_{q_i}$ as the kernel of
\[
  E_{p_i} \to (E_{p_i}^\perp + F_{q_i})/F_{q_i}.
\]

Finally, for \eqref{e.prohibitC2}, consider the case of a single $q_i=q$; if $E_p$ and $F_q$ are isotropic subspaces such that $\dim(E_p\cap F_q)=k$, then for generic subspaces $F_q \supset F_{q+1} \supset \cdots \supset F_{q+k}$, one has
\begin{equation*}\label{e.equal}
  E_p^\perp \cap F_q^\perp = E_p^\perp \cap F_{q+1}^\perp = \cdots = E_p^\perp\cap F_{q+k}^\perp.
\end{equation*}
For any isotropic $E_{p'}\supseteq E_p$, it follows that
\[
  E_{p'} \cap F_q^\perp = E_{p'} \cap F_{q+1}^\perp = \cdots = E_{p'}\cap F_{q+k}^\perp,
\]
so we should only impose rank conditions on the first of these.  Writing $F_q^\perp = F_{-q+1}$, $F_{q+1}^\perp = F_{-q}$, etc., this means the values $-q,-q-1,\ldots,-q-k+1$ should be prohibited.  Accounting for intersections previously imposed leads to the above condition \eqref{e.prohibitC2}.
\end{rmk}

\begin{rmk}
A type C triple determines a signed permutation $w(\triple)$, as follows.  Starting in position $p_1$, first place $k_1$ consecutive integers in increasing order, ending with $\bar{q}_1$.  Then starting in position $p_2$ (or the next available position to the right of $p_2$), place $k_2-k_1$ integers, consecutive among those whose absolute values have not been used, ending in at most $\bar{q}_2$.  Continue until $k_s$ numbers have been placed, and finish by filling in the gaps with the smallest available positive integers.

The conditions on $\triple$ can be understood combinatorially: they guarantee that the length of the signed permutation $w(\triple)$ is equal to the size of the partition $\lambda(\triple)$; this can be checked directly from the construction, using the combinatorial characterization of length from \cite[\S8.1]{BB}.  They also ensure that $\rho_{k_i}$ is the number of entries of $w$ which are less than $-|q_i|$; that the $k_\aaa$ entries of $w$ placed in the first $\aaa$ steps of the above recipe are all negative; and that the entries placed after the first $\aaa$ steps are all positive.

For example, consider $\triple = (\, 1\,3\,5\,6\,7\,9\, ,\; 9\,7\,6\,5\,2\,2\, ,\; 6\,3\,\bar{2}\,\bar{5}\,\bar{7}\,\bar{9}\, )$.  The corresponding signed permutation is built in seven steps:
\begin{align*}
   & \cdot\;\cdot\;\cdot\;\cdot\;\cdot\;\cdot\;\cdot\;\cdot\;{\bf\bar{6}}\;\cdot\;, \\
   &  \cdot\;\cdot\;\cdot\;\cdot\;\cdot\;\cdot\;{\bf\bar{4}\;\bar{3}}\;{\bar{6}}\;\cdot\;, \\
   & \cdot\;\cdot\;\cdot\;\cdot\;\cdot\;{\bf 1}\;{\bar{4}\;\bar{3}}\;{\bar{6}}\;{\bf 2}\;, \\
   & \cdot\;\cdot\;\cdot\;\cdot\;{\bf 5}\;{ 1}\;{\bar{4}\;\bar{3}}\;{\bar{6}}\;{ 2}\;, \\
   & \cdot\;{\bf 7}\;\cdot\;\cdot\;{5}\;{ 1}\;{\bar{4}\;\bar{3}}\;{\bar{6}}\;{ 2}\;, \\
   & \cdot\;{7}\;{\bf 8\; 9}\;{5}\;{ 1}\;{\bar{4}\;\bar{3}}\;{\bar{6}}\;{ 2}\;, \\
  w(\triple) =\;& {\bf 10}\;{7}\;{ 8\; 9}\;{5}\;{ 1}\;{\bar{4}\;\bar{3}}\;{\bar{6}}\;{ 2}\;.
\end{align*}
Since $\aaa=2$, the $k_\aaa=3$ negative entries all appear in the first and second steps.  The partition $\lambda(\triple)$ is the one displayed in Figure~\ref{f.shape}, and a computation shows that $\ell(w(\triple)) = |\lambda(\triple)| = 50$.

It would be interesting to know more about the combinatorial properties of signed permutations arising this way.  For example, are they characterized by pattern avoidance?\footnote{Jordan Lambert Silva has recently found a pattern-avoidance criterion for these signed permutations \cite{jordan}.}
\end{rmk}

\begin{rmk}\label{r.kstrict}
When all $p_i=p$, then a $\rho$-strict partition $\lambda(\triple)$ is one so that all parts of size greater than $p-1$ are distinct; that is, it is a {\em $(p-1)$-strict partition} as defined in \cite{BKT1}.  Geometrically, all $E_{p_i}$ have rank $n+1-p$, so the locus comes from an isotropic Grassmannian.  Conversely, given a $(p-1)$-strict partition $\lambda$, following \cite{BKT1} define
\[
  P_j(\lambda) = n+p-1+j-\lambda_j-\#\big\{ i<j\,\big|\, \lambda_i+\lambda_j > 2p-2+j-i \big\}.
\]
If we define a type C triple by setting $k_j=j$, $p_j=p$, and
\begin{align*}
   q_j &= \begin{cases} n-P_j &\text{ when }P_j>n; \\ n+1-P_j &\text{ when }P_j\leq n,\end{cases}
\end{align*}
one recovers $\lambda=\lambda(\triple)$.  Indeed, one can check that $\rho(\triple)$ is given by
\begin{align*}
  \rho_j &= \#\big\{ i<j\,\big|\, q_i+q_j>0 \big\} \\
         &= \#\big\{ i<j\,\big|\, \lambda_i+\lambda_j>2p-2+j-i \big\} 
\end{align*}
in this situation.

Similarly, for a triple $\triple$ with all $p_i=p$ and $k_i=i$, the {\em characteristic index} $\chi$ used in \cite{IM} is given by $\chi_i = q_i-1$ for $i\leq a$, and $\chi_i=q_i$ for $i>a$.
\end{rmk}

\section{Type B: odd orthogonal bundles} \label{s.typeB}

A triple of type B is the same as in type C, as are the definitions of the sequence $\rho(\triple)$, the partition $\lambda(\triple)$, the raising operator $R^{(\rho,\ell)}$, and the theta-polynomial.

The geometry starts with a vector bundle $V$ of rank $2n+1$, equipped with a nondegenerate quadratic form.  We have two flags of subbundles,
\begin{align*}
  E_{p_1} \subset E_{p_2} \subset \cdots \subset E_{p_s} \subset V, \\
  F_{q_1} \subset F_{q_2} \subset \cdots \subset F_{q_s} \subset V;
\end{align*}
the isotropicity and ranks of these are exactly as in type C.  The degeneracy locus is
\[
  \Omega_\triple = \{ x\in X \,|\, \dim(E_{p_i} \cap F_{q_i}) \geq k_i \text{ for } 1\leq i\leq s \}.
\]

Let $M$ be the line bundle $\det V$, and note that
\[
  M \isom F^\perp/F,
\]
for any maximal isotropic $F\subset V$.  In fact, we have $M\isom \det( D^\perp/D )$ for any isotropic $D\subset V$.

Given a triple, let $\ell=k_s$ and $r=k_\aaa$, recalling that $\aaa$ is the index such that $q_{\aaa}>0>q_{\aaa+1}$.  For $i\leq \aaa$ let $c(k_i) = c(V-E_{p_i}-F_{q_i}-M)$, and for $i>\aaa$, let $c(k_i) = c(V-E_{p_i}-F_{q_i})$.  As before, when $k_{i-1}<k\leq k_i$, we set $c(k)=c(k_i)$.

\begin{thm}\label{t.typeB}
We have ${2^{r}}\,[\Omega_\triple] = \Theta^{(\rho(\triple))}_{\lambda(\triple)}(c(1),c(2),\ldots,c(\ell))$.
\end{thm}

The four steps of the proof are almost the same as in type C.  We will indicate the differences.

\subsection{Basic case}

Take $s=1$, $k_1=\ell=1$, and $p_1=n$, so $E_n$ is a line bundle and $\Omega_\triple$ is the locus where $E_n\subseteq F_{q_1}$.  When $q_1>0$, so $F_{q_1}$ is isotropic, the Proposition of Appendix~\ref{s.B} gives
\[
  2\,[\Omega_\triple] = c_{n+q_1-1}(V-F_{q_1}-E_n-M).
\]
On the other hand, when $q_1<0$, so $F_{q_1}$ is coisotropic, the locus is defined (scheme-theoretically) by the vanishing of $E_n \to V/F_{q_1}$, and
\[
  [\Omega_\triple] = c_{n+q_1}(V/F_{q_1}\otimes E_n^*) = c_{n+q_1}(V-F_{q_1}-E_n)
\]
in this case.

\subsection{Dominant case}
Now $k_i=i$ and $p_i=n+1-i$, for $1\leq i\leq s$, and write $D_i = E_{p_i}$.  As in type C, we have a filtration by $Z_j$, the locus where $D_i \subseteq F_{q_i}$ for all $i\leq j$, so that $Z_0=X$ and $Z_s=\Omega_\triple$.  When $j\leq \aaa$ (so $q_j>0$), the basic case says that
\[
  2\,[Z_j] = [Z_{j-1}]\cdot c_{\lambda_j}(D_{j-1}^\perp/D_{j-1}-F_{q_j}/D_{j-1}-D_j/D_{j-1} -M);
\]
and when $j> \aaa$ (so $q_j<0$),
\[
  [Z_j] = [Z_{j-1}]\cdot c_{\lambda_j}(D_{\rho_j}^\perp/D_{j-1}-F_{q_j}/D_{j-1}-D_j/D_{j-1});
\]
where in each case $\rho_j$ and $\lambda_j$ is defined as in type C.  
We therefore have
\begin{align*}
 2^{\aaa}\, [\Omega_\triple] & = \left( \prod_{j=1}^\aaa c_{\lambda_j}(D_{\rho_j}^\perp/D_{j-1} - F_{q_j}/D_{j-1} - D_j/D_{j-1} - M) \right) \\
   & \qquad \times \left( \prod_{j=\aaa+1}^s c_{\lambda_j}(D_{\rho_j}^\perp/D_{j-1} - F_{q_j}/D_{j-1} - D_j/D_{j-1}) \right),
\end{align*}
as before.

\medskip

The rest of the proof proceeds exactly as in type C. \qed

\bigskip

As in type C, we recover a Pfaffian formula for vexillary signed permutations.

\begin{cor*}
If all $q_i>0$, then
\[
  2^\ell\, [\Omega_\triple] = \Pf_\lambda( c(1), \ldots, c(\ell) ).
\]
\end{cor*}

\section{Type D: even orthogonal bundles} \label{s.typeD}

A \define{triple of type D} is $\triple=(\bk,\bp,\bq)$, with 
\begin{align*}
 0<&k_1<k_2<\cdots<k_s , \\
 &p_1\geq p_2 \geq \cdots \geq p_s \geq 0, \\
 &q_1 \geq q_2 \geq \cdots \geq q_s.
\end{align*}
The value $q=-1$ is prohibited, and if $p_s=0$, then all $q_i\geq 0$.  Set $\aaa=\aaa(\triple)$ to be the integer such that $q_{\aaa}> -1>q_{\aaa+1}$.

\setcounter{tempc}{\value{equation}}
\setcounter{equation}{0}
\renewcommand{\theequation}{d\arabic{equation}}
A quick way to characterize the further requirements on a type D triple is as follows.  Form $\triple^+$ by replacing each $p_i$ in $\triple$ with $p_i+1$, and replacing each $q_i\geq 0$ in $\triple$ with $q_i+1$; then a type D triple is one such that $\triple^+$ is a type C triple, but with the extra requirement that $\triple^+$ must satisfy \eqref{e.tripleC3} with strict inequality.

To be completely clear, we will spell out the conditions.  Their geometric explanations are analogous to those for type C.  For $i\leq \aaa$, 
\begin{align}
  k_{i}-k_{i-1} &\leq (p_{i-1}-p_{i}) + (q_{i-1}-q_{i}), 
  \label{e.tripleD1}
\end{align}
and as before, this is the only condition when all $q_i$ are nonnegative.

For each $i\leq \aaa$, let $m(i) =  \min\{ m \,|\, q_i+(k_i-k_{i-1})\geq q_{m}\}$.  The negative values
\begin{equation}\label{e.prohibitD2}
  -q_i-1,\, -q_i-2,\, \ldots,\, -q_i-(k_i-k_{m(i)-1})
\end{equation}
are prohibited as values of $q_j$ for $j> \aaa$.

The sequence $\rho(\triple)$ is defined similarly as in type C.  Set $\rho_k=k-1$ for $k\leq k_\aaa$.  For $i> \aaa$, set $\rho_{k_i}=k_j$, where $j$ is the index such that $q_j\geq-q_i>q_{j+1}+1$.  Then fill in the other parts minimally subject to $\rho_{k_\aaa+1}\geq \cdots \geq \rho_{k_s}\geq 0$.  We require
\begin{equation}\label{e.tripleD3}
  k_i - \rho_{k_i} < -q_i
\end{equation}
for all $i>\aaa$.

Finally, for $i>\aaa+1$,
\begin{equation}\label{e.tripleD4}
 (k_i-k_{i-1}) + (\rho_{k_{i-1}}-\rho_{k_i}) \leq (p_{i-1}-p_i) + (q_{i-1}-q_i),
\end{equation}
with equality implying either $\rho_{k_{i-1}}=\rho_{k_i}$ or $k_i=k_{i-1}+1$, 
and
\begin{equation}\label{e.tripleD5}
  k_s \geq -p_s-q_s+\rho_{k_s}.
\end{equation}

\renewcommand{\theequation}{\arabic{equation}}
\setcounter{equation}{\value{tempc}}


The associated partition $\lambda(\triple)$ is defined by
\begin{align*}
 \lambda_{k_i} = \begin{cases} p_i+q_i  & \text{if } i\leq \aaa; \\ p_i+q_i+k_i-\rho_{k_i}  &\text{if } i> \aaa;\end{cases}
\end{align*}
 filling in the other parts minimally subject to
\begin{align*}
 \lambda_1 > \cdots > \lambda_{k_\aaa} \geq \lambda_{k_\aaa+1} \geq \cdots \geq \lambda_\ell \geq 0,
\label{en.tripleDla}
\end{align*}
where $\ell=k_s$.  
That is, $\lambda(\triple)$ is a $\rho(\triple)$-strict partition.\footnote{In the case where $k_i=i$ and $p_i=p$ for all $i$, the $\rho$-strict partition $\lambda(\triple)$ constructed from a type D triple is $p$-strict in the sense of \cite{BKT2}.  The ``type'' defined in [{\it op.~cit.}] arises from geometry.  Fix a maximal isotropic bundle $E$ containing the bundle $E_p$.  When no part of $\lambda$ is equal to $p$, the type is $0$.  If some part of $\lambda$ is equal to $p$, then some $q_i=0$, and in this case, the type is defined to be $1$ or $2$ depending on whether $n+\dim(E\cap F_0)$ is odd or even.  See the remark at the end of this section.}  (In contrast to type C, we allow $\lambda_{k_\aaa}=\lambda_{k_\aaa+1}$.)

A key difference in type D is that $\lambda(\triple)$ may have a part equal to $0$, and this is included in the data.  (For example, this happens when $p_s=q_s=0$; in this case, the total number of parts determines the dimension of the intersection of two maximal isotropic subspaces.)

The raising operators and polynomials require some setup; details are explained in Appendix~\ref{s.A}.  We will have elements $c(i) = d(i)+e(i)$, so that $c(i)_k = d(i)_k+e(i)_k$.  We will also have operators $\delta_i$, which acts on a monomial $c(1)_{\alpha_1} \cdots c(\ell)_{\alpha_\ell}$ by replacing $c(i)_{\alpha_i}$ with $d(i)_{\alpha_i}$; that is, $\delta_i$ sends $e(i)$ to zero and leaves everything else unchanged.

Given integers $0\leq r \leq \ell$, and a unimodal sequence of nonnegative integers $\rho = (\rho_1,\ldots,\rho_\ell)$ with $\rho_j<j$ and $\rho_j=j-1$ for $1\leq j\leq r$, we define the raising operator
\[
  \tilde{R}^{(\rho,\ell)} =  \left(\prod_{1\leq i<j\leq \ell} (1-\tilde{R}_{ij}) \right)\left(\prod_{1\leq i\leq \rho_j < j \leq \ell} (1+\tilde{R}_{ij})^{-1}\right) ,
\]
where $\tilde{R}_{ij} = \delta_i\delta_j R_{ij}$.  Using these, the \define{eta-polynomial} for a $\rho$-strict partition is defined as
\[
  \Eta^{(\rho)}_\lambda( c(1),\ldots,c(\ell) ) = \tilde{R}^{(\rho,\ell)}\cdot (c(1)_{\lambda_1}\cdots c(\ell)_{\lambda_\ell}).
\]

Here is the geometry.  We have a vector bundle $V$ of rank $2n$, equipped with a nondegenerate quadratic form taking values in the trivial bundle.  There are flags of subbundles,
\begin{align*}
  E_{p_1} \subset E_{p_2} \subset \cdots \subset E_{p_s} \subset V, \\
  F_{q_1} \subset F_{q_2} \subset \cdots \subset F_{q_s} \subset V;
\end{align*}
each $E_p$ has rank $n-p$ and is isotropic; and each $F_q$ has rank $n-q$, and is isotropic when $q\geq 0$ and coisotropic when $q<0$.  Note that $F_q^\perp = F_{-q}$.

The degeneracy locus is defined as before, by
\[
  \Omega_\triple = \{ x\in X \,|\, \dim(E_{p_i} \cap F_{q_i}) \geq k_i \text{ for } 1\leq i\leq s \}.
\]
The usual type D caveat applies: this acquires its scheme structure via pullback from a Schubert bundle in a flag bundle, and even there, it must be taken to mean the closure of the locus where equality holds (see, e.g., \cite[\S6]{FP} or \cite[\S6.3.2]{T2}).

Now given a type D triple, set $\ell=k_s$ and $r=k_\aaa$.  Let $d(k_i)=c(V-E_{p_i}-F_{q_i})$, and for $i\leq \aaa$ (so $q_i\geq 0$), set $e(k_i) = e(E_{p_i},F_{q_i})$, where the latter is defined as
\[
 e(E_{p_i},F_{q_i}) := (-1)^{\dim(E\cap F)}\,c(E/E_{p_i}+F/F_{q_i}),
\]
for some maximal isotropic bundles $E\supseteq E_{p_i}$ and $F\supseteq F_{q_i}$.  (Only the Euler class $e_{p_i+q_i}(E_{p_i},F_{q_i})$ appears in our formulas, and this is independent of the choice of such maximal $E$ and $F$.)  When $i> \aaa$, we set $e(k_i)=0$; and as usual, when $k_{i-1}<k\leq k_i$, we set $d(k)=d(k_i)$ and $e(k)=e(k_i)$.  Finally, let
\[
  c(k) = d(k) + (-1)^{k}e(k).
\]

\begin{thm}\label{t.typeD}
We have ${2^{r}}\, [\Omega_\triple] = \Eta^{(\rho(\triple))}_{\lambda(\triple)}(c(1),c(2),\ldots,c(\ell))$.
\end{thm}

Most of the proof proceeds exactly as in type B.  We will go through the outline briefly, to point out the differences.

\subsection{Basic case}

Here $s=1$, $k_1=\ell=1$, and $p_1=n-1$, so $E_{n-1}$ is a line bundle and $\Omega_\triple$ is the locus where $E_{n-1}\subseteq F_{q_1}$.  Just as before, we have
\begin{align*}
  2\,[\Omega_\triple] &= c_{n+q_1-1}(V-F_{q_1}-E_{n-1}) - e_{n+q_1-1}(E_{n-1},F_{q_1}) &\text{when } q_1\geq 0; \\
  [\Omega_\triple] &= c_{n+q_1}(V-F_{q_1}-E_{n-1})  &\text{when } q_1 <-1.
\end{align*}
The proof is the same as in type B.  

\subsection{Dominant case}

Now $k_i=i$ for $1\leq i\leq s$, and $p_i=n-i$, so $D_i = E_{n-i}$ has rank $i$.  Let $Z_j$ be the locus where $D_j \subseteq F_{q_j}$, so $Z_0=X$ and $Z_s=\Omega_\triple$.  When $j\leq \aaa$, applying the basic case with $V$ replaced by $D_{j-1}^\perp/D_{j-1}$, we obtain
\begin{align*}
  2\, [Z_j] &= [Z_{j-1}]\cdot \left( c_{\lambda_j}( D_{j-1}^\perp/D_{j-1} - F_{q_j}/D_{j-1} - D_j/D_{j-1} )\right. \\
  & \qquad \left. - e_{\lambda_j}( D_j/D_{j-1}, F_{q_j}/D_{j-1} ) \right) \\
          &= [Z_{j-1}] \cdot \left( c_{\lambda_j}( D_{j-1}^\perp/D_{j-1} - F_{q_j}/D_{j-1} - D_j/D_{j-1} ) \right. \\
          & \qquad \left. + (-1)^{j} e_{\lambda_j}(D_j,F_{q_j})\right) ,
\end{align*}
where $\lambda_j = n-j+q_j$.

When $j> \aaa$, using $D_{j-1}^\perp/(F_{q_j}\cap D_{j-1}^\perp) = D_{\rho_j}^\perp/F_{q_j}$ (by \eqref{e.geo-rho} as in type C), the locus is given by the vanishing of $D_j/D_{j-1} \to D_{\rho_j}^\perp/F_{q_j}$, so the basic case says
\[
  [Z_j] = [Z_{j-1}] \cdot c_{\lambda_j}( D_{\rho_j}^\perp/D_{j-1} - F_{q_j}/D_{j-1} - D_j/D_{j-1} ),
\]
where $\rho_j$ and $\lambda_j$ are defined as above. 
(Here, the strict inequality \eqref{e.tripleD3} ensures that $q<-1$ in this application of the basic case.)

It follows that
\begin{align*}
  2^{\aaa}\, [\Omega_\triple] & = \left( \prod_{j=1}^{\aaa} \left( c_{\lambda_j}(D_{\rho_j}^\perp/D_{j-1} - F_{q_j}/D_{j-1} - D_j/D_{j-1}) +(-1)^{j}e_{\lambda_j}(D_j,F_{q_j}) \right) \right)  \\
   & \qquad \times \left( \prod_{j=\aaa+1}^s c_{\lambda_j}(D_{\rho_j}^\perp/D_{j-1} - F_{q_j}/D_{j-1} - D_j/D_{j-1}) \right).
\end{align*}
This leads to
\begin{equation}\label{e.typeDop}
  2^{\aaa}\, [\Omega_\triple] = \tilde{R}^{(\rho,s)} c(1)_{\lambda_1} \cdots c(s)_{\lambda_s},
\end{equation}
with $c(j) = d(j) + (-1)^{j} e(j)$ for $j\leq \aaa$ and $c(j) = d(j)$ for $j> \aaa$, as defined above.  The deduction of \eqref{e.typeDop} differs slightly from the previous cases (e.g., deducing \eqref{e.typeCop} from \eqref{e.typeCt}): one must verify
\begin{align*}
 & t_i^b\cdot \left( c_{\lambda_i}(D_{i-1}^\perp/D_{i-1} - F_{q_i}/D_{i-1} - D_i/D_{i-1}) +(-1)^{i}e_{\lambda_i}(D_i,F_{q_i}) \right) \\
   &\hspace{3in} = c_{\lambda_i+b}(D_{i-1}^\perp/D_{i-1} - F_{q_i}/D_{i-1} - D_i/D_{i-1})
\end{align*}
for $i\leq \aaa$, where $t_i=-c_1(D_i/D_{i-1})$, since this shows that $t_i$ acts as $\delta_iT_i$ on the factor $c(i)_{\lambda_i}$.  To do this, by replacing $D_{i-1}^\perp/D_{i-1}$ with $V$ one may reduce to the case $i=1$; one also reduces to the case $b=1$ by applying Identity~\eqref{ident2}.  In this case, we compute
\begin{align*}
 &t_1\cdot \left( c_{n-1+q_1}(V-F_{q_1}-D_1) - e_{n-1+q_1}(D_1,F_{q_1}) \right) \\
 &\qquad= \sum_{k=1}^{n+q_1} t_1^k\, c_{n+q_1-k}(V/F_{q_1}) 
   + (-1)^{\dim(E\cap F)}c_{n+q_1}(E+F/F_{q_1}) \\
 &\qquad = c_{n+q_1}(V-D_1-F_{q_1}) - c_{n+q_1}(V/F_{q_1}) 
  + (-1)^{\dim(E\cap F)}c_{n}(E)\, c_{q_1}(F/F_{q_1}) \\
 &\qquad = c_{n+q_1}(V-D_1-F_{q_1})
   + c_{q_1}(F/F_{q_1})\big( (-1)^{\dim(E\cap F)}c_{n}(E) - c_n(F^*)\big),
\end{align*}
and apply the relation $(-1)^{\dim(E\cap F)}c_{n}(E) = c_n(F^*)$, due to Edidin and Graham \cite{EG}.

\smallskip

The remainder of the proof proceeds as in the other types.  To deduce the general case, one needs to apply \S\ref{ss.A.theta}, Lemma~\ref{l.inflate2}.  The relations in the hypothesis of that lemma require that
\[
 (d(k)_{m}-(-1)^\ell e(k)_{m})( d(k)_{m}+(-1)^\ell e(k)_{m} ) + 2\sum_{j>0} (-1)^j d(k)_{m+j}\, d(k)_{m-j}
\]
vanish for all $m\geq \lambda_k$.  This expression may be re-written as
\begin{align*}
 & \sum_{j=-\infty}^\infty (-1)^j d(k)_{m+j} d(k)_{m-j} 
  - (e(k)_{m})^2 \\
 &\quad = (-1)^{m}\,c_{2m}( E_{p_i}^\perp/E_{p_i} + (F_{q_i}^\perp/F_{q_i})^* ) \\ &\qquad - c_{m}( E/E_{p_i} + F/F_{q_i})^2,
\end{align*}
which vanishes if $m>\lambda_k=p_i+q_i$, since this is the rank of $E/E_{p_i}+F/F_{q_i}$.  When $m=\lambda_k$, we have $(-1)^{\lambda_k}\,c_{2\lambda_k}( E_{p_i}^\perp/E_{p_i} + (F_{q_i}^\perp/F_{q_i})^* ) = c_{\lambda_k}( E/E_{p_i}+F/F_{q_i} )^2$, so the relation holds in this case as well.\qed

\medskip

To extract a Pfaffian from the case where $\aaa=s$, so all $q_i\geq0$, one needs a little algebra, given by the Theorem of \S\ref{ss.A.raising}.  Applying this proves the following:

\begin{cor*}
If all $q_i\geq0$, then
\begin{align}\label{e.corD}
  2^\ell\, [\Omega_\triple] = \Pf_\lambda( c(1), \ldots, c(\ell) ).
\end{align}
\end{cor*}
\noindent
Unpacking the definition of the $c(i)$'s, the right-hand side is the Pfaffian of the matrix $(m_{ij})$, with
\begin{align*}\label{e.mtxD}
  m_{ij} &= (d(i)_{\lambda_i}-(-1)^\ell e(i)_{\lambda_i})( d(j)_{\lambda_j}+(-1)^\ell e(j)_{\lambda_j}) \\ &\qquad + 2\sum_{t>0} (-1)^t d(i)_{\lambda_i+t}\, d(j)_{\lambda_j-t},
\end{align*}
for $1\leq i<j\leq \ell$, and $m_{0j}=d(j)_{\lambda_j}+e(j)_{\lambda_j}$ for $0< j\leq \ell$ if $\ell$ is odd.

\begin{rmk*}
A Schubert variety in an orthogonal Grassmannian $OG(n-p,2n)$ is defined by a triple $\triple$ with all $p_i=p$.  To obtain these as degeneracy loci according to our setup, one should use two different maximal isotropic spaces in the reference flag $F_\bullet$.  In fact, given a complete isotropic flag, there is a unique maximal isotropic subspace $F_1\subset F_0' \subset F_{\bar{1}}$ which is distinct from $F_0$; both $F_0$ and $F_0'$ must be used to define Schubert varieties.

For example, consider $\triple = ( 1,\, p,\, 0 )$, so $\lambda = (p)$.  In the setup of \cite{BKT2}, there are two Schubert varieties whose $p$-strict partition is  $\lambda$, given by $\dim(E_p\cap F_0)\geq 1$ and $\dim(E_p\cap F_0')\geq 1$, respectively.  (When $p=n-1$, these are the two maximal linear spaces in the quadric.)  The respective formulas are $c_p(V-E_p-F_0)-e_p(E_p,F_0)$ and $c_p(V-E_p-F_0')-e_p(E_p,F_0')$.  Note that $e_p(E_p,F_0') = -e_p(E_p,F_0)$.  (Compare \cite[Example~A.3]{BKT2}.)

The strict inequality of \eqref{e.tripleD3} excludes some Grassmannian Schubert loci, however.  For example, the $2$-strict partition $\lambda = (4,2,2)$ should correspond to a triple $(\,1\,2\,3\,,\, 2\,2\,2 \,, \, 2\, 0\, \bar{2}\,)$, but this has $k_3-\rho_{k_3}=2=-q_3$, violating \eqref{e.tripleD3}.
\end{rmk*}

\appendix
\section{Algebra of Pfaffians and raising operators}\label{s.A}

\addtocontents{toc}{\setcounter{tocdepth}{2}}

\renewcommand{\thesubsection}{A.\arabic{subsection}}
\renewcommand{\theequation}{A.\arabic{equation}}
\setcounter{thm}{0}
\setcounter{equation}{0}

\subsection{A Pfaffian identity}\label{ss.A.pf}

Given $a_{ij}$ in a commutative ring $A$, for $1 \leq i < j \leq n$, and  $n$  even, we denote by $\Pf(a_{ij})$ the Pfaffian of the skew-symmetric matrix $(a_{ij})$ with entries  $a_{ij}$ for $i < j$, and $a_{i i} = 0$ and $a_{ij} = - a_{ji}$  for  $i > j$.  That is, for $n = 2m$, 
\begin{equation}\label{PfaffianDef}
\Pf(a_{ij}) = \sum \pm a_{i_1 j_1} a_{i_2 j_2} \dots a_{i_m j_m} ,
\end{equation}
the sum over all permutations  $i_1 j_1 i_2 j_2 \dots i_m j_m$ of $1 2 \dots n$, with  $i_1 < i_2 < \dots < i_m$ and $i_r < j_r$ for all $r$, the sign being the sign of the permutation.  The same notation is used whenever $1 2 \dots n$ is replaced by any set consisting of an even number of integers in increasing order.  For example, the expansion along the first row can be written 
\[
\Pf(a_{ij})  =  \sum_{k = 2}^n (-1)^{k-1}   a_{1 k}  \Pf(a_{ij})_{\hat{1},\hat{k}},
\]
where the hats denote that the integers are taken from the first $n-2$ positive integers, omitting $1$ and $k$.  

We will often want formulas for odd as well as even $n$.  For this, when $n$ is odd, we will use the integers from $0$ to $n$.  In addition to the $a_{ij}$ for $1 \leq i < j \leq n$ we also need to specify $a_{0j}$  for $1 \leq j \leq n$.  Then the Pfaffian is given by the identity
\[  
\Pf(a_{ij})  =  \sum_{k = 1}^n (-1)^{k-1}  a_{0 k}  \Pf(a_{ij})_{\hat{k}},
\]
with $i < j$ taken from positive integers not equal to $k$.

Assume now that the ring $A$ contains elements $\delta_1, \ldots , \delta_n$  satisfying $\delta_i^2 = \delta_i$ for all $i$.  
Set $\epsilon_i = 2 \delta_i - 1$, so $\epsilon_i^2 = 1$ and $\epsilon_i \delta_i = \delta_i$ for all $i$.  (The classical case is when $\delta_i = 1$, so $\epsilon_i = 1$, for all $i$.)  Set $\delta_0 = 1$.

Let $T_1, \ldots, T_n$ be indeterminates, and let $B$ be the localization of $A[T_1, \ldots, T_n]$ at the multiplicative set of non-zero-divisors (which includes all $T_j$, and all $T_j - \delta_i\delta_jT_i$ for all $i < j$).  Set, for $1 \leq i < j \leq n$,
\[
H_{ij}  =  \frac{T_j - \delta_i\delta_jT_i}{T_j + \delta_i\delta_jT_i}.
\]
Also set $T_0 = 0$, and $H_{0j} = 1$ for $1 \leq j \leq n$.  

Our goal is to write the product $\prod_{1 \leq i < j \leq n} H_{ij}$ as a Pfaffian.  The classical case is due to Schur: 
\begin{equation}\label{SchurPfaffian}
\prod_{1 \leq i < j \leq n}  \frac{T_j - T_i}{T_j + T_i} = \Pf\left(\frac{T_j - T_i}{T_j + T_i}\right),
\end{equation}
where, if $n$ is odd, the $(0,j)$ entry of the matrix is $1$, for $1 \leq j \leq n$.
Our generalization is:

\begin{thm*}
Set $a_{ij} = \epsilon_i^{n-i+1} \epsilon_j^{n-j} H_{ij}$, for $1 \leq i < j \leq n$, and 
set $a_{0j} = \epsilon_j^{n-j}$ for $1 \leq j \leq n$.  Then 
\[
\prod_{1 \leq i < j \leq n} H_{ij} = \Pf(a_{ij})
\]
\end{thm*}

We will deduce this from a result of Knuth \cite{Knuth}, as simplified by Kazarian \cite{K}.  Following their notation, define, for $x < y$ nonnegative integers, 
\[ 
f[xy] = \epsilon_x H_{xy} = \epsilon_x \frac{T_y - \delta_x\delta_y T_x}{T_y + \delta_x\delta_y T_x}.
\]
Note that $f[0y] = \epsilon_0 H_{0y} = 1$.  Set $f[xx] = 0$ and $f[xy] = -f[yx]$ for $x > y \geq 0$.  For a word $\alpha = x_1 \dots x_n$, with each $x_i$ a nonnegative integer, define $f[\alpha]$ to be the Pfaffian of the matrix whose $(i,j)$ entry is $f[x_ix_j]$, for $n$ even; for $n$ odd, define $f[\alpha]$ to be $f[0\alpha]$.  Note that $f[\alpha]$ vanishes if two letters in $\alpha$ coincide, and it changes sign if the positions of two letters are interchanged. 

\begin{prop*} 
For all $n \geq 2$, and nonnegative integers $x_1, \ldots, x_n$, 
\[
\prod_{1 \leq i < j \leq n}  f[x_i x_j]  =  \Pf(f[x_i x_j]).
\]
\end{prop*}

\begin{proof}  By \cite{Knuth} and \cite{K}, this identity holds for all $n$ if it holds for $n = 3$.\footnote{This follows by an induction on $n$ from the Tanner identity 
\[
f[\alpha]f[\alpha w x y z] = f[\alpha w x] f[\alpha y z] -  f[\alpha w y] f[\alpha x z]  +  f[\alpha w z] f[\alpha x y].
\]
for any word $\alpha$ and letters $w, x, y$ and $z$ (see \cite{Knuth} (1.1) and (4.2)).}  
For $n = 3$, it asserts that, for $x < y < z$ positive integers, 
\[
\epsilon_x^2\epsilon_y H_{xy} H_{xz} H_{yz} = \epsilon_y H_{yz} - \epsilon_x H_{xz} + \epsilon_x H_{xy},
\]
or, $H_{xy} H_{xz} H_{yz}  = H_{yz} - \epsilon_x \epsilon_y H_{xz} + \epsilon_x \epsilon_y H_{xy}$. 
Clearing denominators, this amounts to a simple identity among cubic polynomials in the three variables $T_x$, $T_y$, and $T_z$, which is an easy exercise. 
\end{proof}

\begin{cor*}
Let $1 \leq x_1 < x_2 < \dots < x_n$.  Set $a_{ij} = \epsilon_{x_i} H_{x_i x_j}$ for $1 \leq i < j \leq n$ and $a_{0 j} = 1$ for $1 \leq j \leq n$.  Then 
\[
\prod_{1 \leq i < j \leq n}   H_{x_i x_j}  =  \prod_{i = 1}^n \epsilon_{x_i}^{n-i} \cdot \Pf(a_{ij}).
\]
Equivalently, setting $b_{ij} = \epsilon_{x_i}^{n-i+1} \epsilon_{x_j}^{n-j} H_{x_i x_j}$ for $1 \leq i < j \leq n$ and $b_{0j} =  \epsilon_{x_j}^{n-j}$ for $1 \leq j \leq n$, 
\[
\prod_{1 \leq i < j \leq n}   H_{x_i x_j}  =  \Pf(b_{ij}).
\]
\end{cor*}

\begin{proof}
The Proposition says that $\prod_{i < j} \epsilon_{x_i} H_{x_i x_j} = \Pf(\epsilon_{x_i} H_{x_i x_j})$, which yields the first statement.  The second follows from the first, using the basic identity
\[
\Pf(\epsilon_i^{m_i}\epsilon_j^{m_j} a_{ij}) = \prod_{i=1}^n \epsilon_i^{m_i} \Pf(a_{ij})
\]
for any $a_{ij}$, $i < j$ and $\epsilon_i$ in the given ring, and nonnegative integers $m_i$.  (This identity follows immediately from the definition (\ref{PfaffianDef}) of the Pfaffian.)
\end{proof}

The Theorem is the special case of the Corollary when $x_i = i$ for $1 \leq i \leq n$.

\subsection{Raising operators}\label{ss.A.raising}

Formula \eqref{SchurPfaffian} can be rewritten
\begin{align}\label{RaisingPfaffian}
\prod_{1 \leq i < j \leq n} \frac{1 - T_i/T_j}{1+T_i/T_j} = \Pf\left(\frac{1 - T_i/T_j}{1+T_i/T_j}\right).
\end{align}
With $T_i/T_j$ interpreted as a raising operator $R_{ij}$, this formula leads to a classical Pfaffian formula for Schur Q-functions (see \cite[\S III.8]{Mac}).  Our goal here is a small generalization, to be applied to types C and B, and a larger one, using the Theorem from \S\ref{ss.A.pf}, to be applied to type D. 

We will take raising operators $R_{ij}$, for $1 \leq i < j \leq n$, to operate on sequences $s = (s_1, \ldots, s_n)$ in $\Z^n$, by raising the $i^{\text{th}}$ index by $1$, and lowering the $j^{\text{th}}$ index by $1$, keeping the others the same:
\[
R_{ij}(s_1, \ldots, s_n) = (s_1, \ldots, s_i + 1, \ldots, s_j - 1, \dots, s_n).
\]
These operators commute with each other, and satisfy the identities $R_{ij}R_{jk} = R_{ik}$ for $i < j < k$.  By a {\em raising operator} we mean any monomial $R = \prod_{i < j} R_{ij}^{m_{ij}}$ in these $R_{ij}$.  Any raising operator acts bijectively on the set $\Z^n$ of $s$'s.  

We will follow the tradition of using raising operators to act on expressions $\sum a_s c_s$, for $c_s$ certain fixed elements of a ring $A$, and the $a_s$ varying elements of $A$, with $R$ taking $\sum a_s c_s$ to $\sum a_s c_{R(s)}$.  Some care needs to be taken here, as the $c_s$ will evaluate to $0$ when any entry $s_i$ of $s$ is negative, but such $s$ need to appear in the expressions in order for the action of the raising operators to be associative and commutative.  (For example, $R_{23}(R_{12}(c_{(1, 0, 1)})) = R_{23}(c_{(2, -1, 1)}) = c_{(2, 0, 0)}$, which is $R_{12}(R_{23}(c_{(1, 0, 1)})) = R_{12}(c_{(1, 1, 0)})$.)  In addition, one wants only finite expressions $\sum a_s c_s$, but one wants to apply infinitely many raising operators, in expressions like $(1 - R_{ij})/(1+R_{ij}) = 1 + 2 \sum_{k > 0} (-1)^k (R_{ij})^k$.  Garsia \cite{Ga} described one way to deal with these problems in another setting.  We offer here a simple alternative, well suited to our situation.  

Let $P \subset \Z^n$ be the set of $s = (s_1, \ldots, s_n)$ satisfying the inequalities
\[
s_k + s_{k+1} + \cdots + s_n \geq 0 \,\,\,\,\,\,\,\,  \text {for} \,\,\,\,\,\,\,\, 1 \leq k \leq n.
\]
The idea is that any $c_s$ can be set equal to $0$ if $s$ is not in $P$, because any raising operator $R$ takes such a $s$ to an $R(s)$ that is also not in $P$.  We will use the following fact, which is easily proved by induction on $n$. 

\begin{lem*}
For any $s \in \Z^n$, there are only finitely many raising operators $R$ such that $R(s)$ is in $P$.
\end{lem*}

Now let $A$ be any commutative ring.  Assume we are given elements $c(i)_r$ in $A$, for $1 \leq i \leq n$ and $r \in \Z$, with $c(i)_r = 0$ if $r < 0$.  For $s \in \Z^n$, we write $c_s$ for $c(1)_{s_1} c(2)_{s_2} \cdots c(n)_{s_n}$.  By an {\em expression} we mean a finite formal sum $\sum_{s \in P} a_s c_s$, with $a_s$ in $A$.  For any raising operator $R$, we define 
$R(\sum a_s c_s) $ to be the sum $\sum_{R(s) \in P} a_s c_{R(s)}$; that is, one applies $R$ to the index $s$ of each $c_s$, but discards the term if $R(s)$ is not in $P$.\footnote{Logically, an expression is a function $a \colon P \to A$, taking $s$ to $a_s$, which vanishes for all but a finite number of $s$; $R(a)$ is the function that takes $s$ to $a_{R^{-1}(s)}$.}  
This gives an action of the polynomial ring $A[R_{ij}]_{1 \leq i < j \leq n}$ on expressions: 
\[
\left(\sum b_R R \right) \left(\sum a_s c_s\right) = \sum_{t \in P} \left( \sum_{R(s) = t} b_R a_s \right) c_t .
\]
By the lemma, this extends to an action of the power series ring  $A[[R_{ij}]]_{1 \leq i < j \leq n}$ on the set of all expressions.\footnote{Traditionally, one allows arbitrarily long sequences $(s_1, \ldots, s_n)$, but also requires that $c(i)_0 = 1$ for all $i$.  Our conventions apply for such sequences, provided that there is an $N$ such that $c(i)_0 = 1$ for $i > N$.}  By the {\em evaluation} of an expression $\sum a_s c_s$ we mean the corresponding element $\sum a_s c(1)_{s_1} \cdots c(n)_{s_n}$ in $A$.

The following is a version of the classical result that suffices for our application to types C and B.  It follows from the identity \eqref{RaisingPfaffian}.

\begin{prop*}
Fix $s = (s_1, \ldots, s_n)$ in $P$.  The evaluation of 
\[
   \left( \prod_{1 \leq i < j \leq n} \frac{1 - R_{ij}}{1 + R_{ij}}\right) \cdot c_{s} 
\]
is the Pfaffian of the matrix whose entries are
\[
m_{ij} = c(i)_{s_i} c(j)_{s_j} + 2 \sum_{k > 0} (-1)^k c(i)_{s_i + k} c(j)_{s_j - k} 
\] 
for $1 \leq i < j \leq n$, with $m_{0 j} = c(j)_{s_j}$ for $1 \leq j \leq n$ when $n$ is odd.
\end{prop*}

For type D we need a strengthening of this proposition, in which each $c(i)_r$ is written as as sum: 
\[
c(i)_r = d(i)_r +  e(i)_r ,
\]
for elements $d(i)_r$ and $e(i)_r$ in $A$, with $d(i)_r = e(i)_r = 0$ for $r < 0$.

New operators $\delta_i$, for $1 \leq i \leq n$ act on these expressions, with $\delta_i$ sending $d(i)_{s_i} +  e(i)_{s_i}$ to $d(i)_{p_i}$, leaving the other factors alone.  That is, $\delta_i$ changes $e(i)$ to $0$.  
Note that the operators $\delta_i$ commute with each other and with the raising operators, and $\delta_i^2 = \delta_i$ for all $i$.  Set $\epsilon_i = 2 \delta_i - 1$; this has the effect of changing $d(i)_{s_i}+e(i)_{s_i}$ to $d(i)_{s_i}-e(i)_{s_i}$. 

\begin{thm*}
Fix $s$ in $P$.  The evaluation of
\[
\left(\prod_{1 \leq i < j \leq n} \frac{1 - \delta_i\delta_jR_{ij}}{1 + \delta_i\delta_jR_{ij}}\right) \cdot c_{s}
\]
is equal to the Pfaffian of the matrix whose entries are
\begin{equation*}
m_{ij} \, = \, (d(i)_{s_i} + (-1)^{n-i+1}e(i)_{s_i}) \cdot (d(j)_{s_j} + (-1)^{n-j}e(j)_{s_j}) 
+ 2 \sum_{k > 0} (-1)^k d(i)_{s_i + k} d(j)_{s_j - k} ,
\end{equation*}
for $1\leq i<j\leq n$, with $m_{0j} = d(j)_{s_j} + (-1)^{n-j} e(j)_{s_j}$ for $1 \leq j \leq n$ when $n$ is odd.  \end{thm*}

\begin{proof}
By the Theorem of \S\ref{ss.A.pf}, 
\[ 
\prod_{1 \leq i < j \leq n} \frac{1 - \delta_i\delta_jR_{ij}}{1 + \delta_i\delta_jR_{ij}}  \, = \, 
\Pf\left(\epsilon_i^{n-i+1}\epsilon_j^{n-j} \frac{1 - \delta_i\delta_jR_{ij}}{1 + \delta_i\delta_jR_{ij}} \right) .
\]
The conclusion follows, since the evaluation of $\epsilon_i^{n-i+1}\epsilon_j^{n-j} \frac{1 - \delta_i\delta_jR_{ij}}{1 + \delta_i\delta_jR_{ij}} $ on $c_{s}$ is $m_{ij} \cdot \prod_{k \neq i, j} c(k)_{s_k}$.
\end{proof}

\begin{ex*}
The case needed for the type D application is when $e(i) = (-1)^{i}\tilde{e}(i)$ for $1 \leq i \leq n$.  In this case
\begin{equation*}
m_{ij} = (d(i)_{s_i} - (-1)^n \tilde{e}(i)_{s_i}) \cdot (d(j)_{s_j} + (-1)^n \tilde{e}(j)_{s_j}) 
+ 2 \sum_{k > 0} (-1)^k d(i)_{s_i + k} d(j)_{s_j - k} ,
\end{equation*}
with $m_{0j} = d(j)_{s_j} + \tilde{e}(j)_{s_j}$.
\end{ex*}

\subsection{A theta-polynomial identity}\label{ss.A.theta}

We need a preliminary identity, which holds in any commutative ring $A$.  For any Laurent series $B=B(t) = \sum_r b_r\, t^r$, with coefficients $b_i$ in $A$, we define $F^B = F^B(t) = \sum F^B_p\, t^p$ by
\[
  F^B(u) = B(-t)\cdot B(t),
\]
with $u=-t^2$.  For any $C=\sum_r c_r\,t^r$, with $(BC)(t) = B(t)\cdot C(t)$, it follows that $F^{BC}(u) = F^B(u)\cdot F^C(u)$, and hence
\begin{equation}\label{e.rel1}
 F^{BC}_r = \sum_{p+q=r} F^B_p \cdot F^C_q.
\end{equation}
In particular, if for some $\lambda'> \lambda$, the relations $F^B_p=0$ and $F^C_q = 0$ hold in $A$ for all $p>\lambda'-\lambda$ and $q\geq \lambda$, then $F^{BC}_r = 0$ for all $r\geq \lambda'$.

We will write $c(i)$ for a collection of elements $c(i)_r$, for $1\leq i\leq \ell$ and $r \in \Z$, and write $c$ for $(c(1), \ldots, c(\ell))$.  If $c(i)=c(j)$, then writing $C(t) = \sum_r c(i)_r\, t^r = \sum_r c(j)_r\, t^r$, we have
\begin{equation}\label{e.rel}
 F^C_p = \left(\frac{1-R_{ij}}{1+R_{ij}}\right)\cdot \left( c(i)_p\,c(j)_p \right) .
\end{equation}

Fix an integer $k$, $0\leq k\leq \ell$, 
and a unimodal sequence $\rho=(\rho_1,\ldots,\rho_\ell)$ such that $\rho_j=j-1$ when $j\leq k$, and $k \geq \rho_{k+1} \geq \cdots \geq \rho_\ell\geq 0$.  A partition $\lambda = (\lambda_1\geq \cdots \geq \lambda_\ell\geq 0)$ is called $\rho$-strict if the sequence $\lambda_1+\rho_1,\ldots,\lambda_\ell+\rho_\ell$ is non-increasing.  
(As with any partition, such a $\lambda$ belongs to the set $P$.)  Given a $\rho$-strict partition $\lambda$, the theta-polynomial is
\[
\Theta^{(\rho)}_{\lambda}(c) = R^{(\rho,\ell)}\cdot c_\lambda,
\] 
where $R^{(\rho,\ell)}$ is the raising operator
\begin{align*}
 R^{(\rho,\ell)} &=  \prod_{1\leq i\leq \rho_j < j\leq \ell} (1+R_{ij})^{-1} \cdot \prod_{1\leq i<j\leq \ell} (1-R_{ij}) .
\end{align*}

\begin{lem}\label{l.inflate}
Fix a $\rho$-strict partition $\lambda$, integers $1\leq m<n\leq \ell$, and elements $b_1,\ldots,b_{n-m}$ in the ring $A$.  Assume $c(m)=c(m+1)=\cdots=c(n)$.  Let $c'(m)=c(m)\cdot(1+b_1+\cdots+b_{n-m})$, and $c'(i)=c(i)$ for $i\neq m$.
\begin{enumerate}[(i)]
\item For $k\leq m<n\leq \ell$, suppose $\lambda_m=\lambda_{m+1}=\cdots=\lambda_{n}$ and also $\rho_m=\cdots=\rho_n$.  Then 
\[
\Theta^{(\rho)}_{\lambda}(c')  = 
\Theta^{(\rho)}_{\lambda}(c).
\] 

\item For $1\leq m<n\leq k$, suppose $\lambda_m = \lambda_{m+1}+1 = \cdots = \lambda_{n}+n-m$ and $\rho_j\not\in\{m,\ldots,n-1\}$ for all $j>n$.  Assume that for each $1\leq i\leq k$, with $C=\sum c(i)_r\,t^r$, the relations $F^C_p=0$ hold for all $p\geq \lambda_i$.  Then
\[
\Theta^{(\rho)}_{\lambda}(c')  = 
\Theta^{(\rho)}_{\lambda}(c),
\] 
and the relations $F^{C'}_p=0$ hold for all $p\geq \lambda_i$, where $C'=\sum c'(i)_r\,t^r$.
\end{enumerate}
\end{lem}

\begin{proof}
For both statements, it is straightforward to reduce to the case $n=m+1$, so we assume this.  

For $m\geq k$, let us write $\lambda_m=\lambda_{m+1}=p$.  The claim is that
\begin{align*}
  R^{(\rho,\ell)} \cdot (\cdots c(m)_{p}\, c(m+1)_{p} \cdots ) &= R^{(\rho,\ell)} \cdot (\cdots c(m)_{p}\, c(m+1)_{p} \cdots ) \\
  & \qquad  + b_1 \cdot R^{(\rho,\ell)} \cdot (\cdots c(m)_{p-1} \,c(m+1)_{p} \cdots ),
\end{align*}
under the assumption $c(m)=c(m+1)$.  We will see that the second term on the right-hand side is zero.

We can write the raising operator as
\begin{align*}
 R^{(\rho,\ell)} &= R' \cdot \left( \frac{\prod_{i=1}^{k-1} (1-R_{i,m})(1-R_{i,m+1})}{\prod_{i=1}^{\rho_m} (1+R_{i,m})(1+R_{i,m+1})}\right) \\
 & \qquad \times \left( \prod_{i=k}^{m-1} \prod_{j=m+2}^\ell (1-R_{i,m})(1-R_{i,m+1})(1-R_{m,j})(1-R_{m+1,j}) \right) (1-R_{m,m+1}),
\end{align*}
where $R'$ involves only $R_{ij}$ such that $i,j\not\in\{m,m+1\}$.  (We used $\rho_m=\rho_{m+1}$ in collecting the denominators of the second factor.)  All factors other than $(1-R_{m,m+1})$ are symmetric in $m$ and $m+1$.  The operator $R_{m,m+1}$ leaves $(\cdots c(m)_{p-1}\,c(m+1)_p \cdots )$ invariant, since $c(m)=c(m+1)$.  It follows that $R^{(\rho,\ell)}\cdot (\cdots c(m)_{p-1}\,c(m+1)_p \cdots ) =0$, as claimed.

For $m<k$, the argument is similar.  We write the raising operator as
\begin{align*}
 R^{(\rho,\ell)} &= R' \cdot \left( \prod_{i=1}^{m-1}\prod_{j=m+2}^\ell (1-R_{i,m})(1-R_{i,m+1})(1-R_{m,j})(1-R_{m+1,j}) \right) \\
 & \qquad \times \left( \mathop{\prod_{j=m+2}^\ell}_{\rho_j\geq m+1} \frac{1}{1+R_{m+1,j}}  \mathop{\prod_{j=m+2}^\ell}_{\rho_j\geq m} \frac{1}{1+R_{m,j}} \right) \cdot \left( \frac{1-R_{m,m+1}}{1+R_{m,m+1}}\right).
\end{align*}
As before, $R'$ does not involve $m$ or $m+1$, so the factors on the first line are clearly symmetric in $m$ and $m+1$.  Since we assume $\rho_j\neq m$ for all $j\geq m+2$, the conditions $\rho_j\geq m+1$ and $\rho_j\geq m$ are equivalent, so the first factor on the second line is also symmetric in $m$ and $m+1$.  To complete the argument, observe that the relation
\[
  \frac{1-R_{m,m+1}}{1+R_{m,m+1}}\cdot c(m)_p\, c(m+1)_p = F_p^C
\]
is anti-symmetric in $m$ and $m+1$, so it remains zero after applying any raising operator which is symmetric in these indices.
\end{proof}

A similar statement holds for eta-polynomials.  These are defined using the raising operator
\[
 \tilde{R}^{(\rho,\ell)} =  \prod_{1\leq i\leq \rho_j < j\leq \ell} (1+\tilde{R}_{ij})^{-1} \cdot \prod_{1\leq i<j\leq \ell} (1-\tilde{R}_{ij}) ,
\]
where $\tilde{R}_{ij} = \delta_i\delta_j R_{ij}$.  
For $1\leq i\leq \ell$, given elements $c(i)_r = d(i)_r+e(i)_r$, with $e(i)=0$ for $i>k$, the eta-polynomial is defined as
\[
\Eta^{(\rho)}_{\lambda}(c) = \tilde{R}^{(\rho,k,\ell)}\cdot c_{\lambda}.
\]

Given such $c(i)$, with $C(t) = \sum c(i)_r\,t^r$ and $D(t)=\sum d(i)_r\,t^r$, let $\tilde{F}^C_p = F^D_p - e(i)_p^2$; if $d(i)=d(j)$ and $e(i)=e(j)$, this is
\[
  \tilde{F}^C_p = \left(\frac{1-\delta_i\delta_j R_{ij}}{1+\delta_i\delta_j R_{ij}}\right) \cdot (c(i)_p \,c(j)_p).
\]

\begin{lem}\label{l.inflate2}
Assume the hypotheses and notation of Lemma~\ref{l.inflate}.

\begin{enumerate}[(i)]
\item With the same hypotheses as in Lemma~\ref{l.inflate}(i), we have
\[
\Eta^{(\rho)}_{\lambda}(c')  = 
\Eta^{(\rho)}_{\lambda}(c).
\]

\item With the hypotheses of Lemma~\ref{l.inflate}(i), assume additionally that $d(m)=d(m+1)=\cdots=d(n)$, so also $e(m)=\cdots=e(n)$.  Assume that for each $1\leq i\leq k$, the relations $\tilde{F}^C_p=0$ and $e(i)_q=0$ hold for all $p\geq \lambda_i$ and $q>\lambda_i$.  Then
\[
\Eta^{(\rho)}_{\lambda}(c')  = 
\Eta^{(\rho)}_{\lambda}(c),
\] 
and the relations $\tilde{F}^{C'}_p=0$ and $e'(i)_q=0$ hold for all $p\geq \lambda_i$ and $q>\lambda_i$, where $C'=\sum c'(i)_r\,t^r$ and $e'(i)=e(i)\cdot b$.
\end{enumerate}
\end{lem}

\noindent
The proof is similar to that of Lemma~\ref{l.inflate}.  For the last statement, concerning the relations, when $i\neq m$ there is nothing to check, since $\tilde{F}^{C'}=\tilde{F}^C$ by definition.  For $i=m$, observe that
\[
  \tilde{F}^{C'}_p = F^{BD}_p -  e'(m)_p^2,
\]
where $e'(m)_p = e(m)_p + b_1 e(m)_{p-1} + \cdots$.  Since $e(m)_q=e(n)_q=0$ for $q>\lambda_n$, and $b_q=0$ for $q>n-m$, we have $e'(m)_p=0$ for $p>\lambda_n+n-m=\lambda_m$, and $e'(m)_{\lambda_m}=b_{n-m} e(m)_{\lambda_m-n+m}$.  Similarly, the relations $\tilde{F}^C_p$ imply $F^D_p=0$ for $p>\lambda_n$, so using \eqref{e.rel1}, we have $F^{BD}_p=0$ for $p>\lambda_m$ and $F^{BD}_{\lambda_m}=b_{n-m}^2 F^D_{\lambda_m-n+m}$.  It follows that $\tilde{F}^{C'}_p = 0$ for $p>\lambda_m$, and $\tilde{F}^{C'}_{\lambda_m} = b_{n-m}^2\tilde{F}^C_{\lambda_m-n+m} = 0$ as well.

\newpage

\section{On quadric bundles}\label{s.B}

\renewcommand{\thesubsection}{B.\arabic{subsection}}
\renewcommand{\theequation}{B.\arabic{equation}}
\setcounter{thm}{0}
\setcounter{equation}{0}


Let $V$ be a vector bundle on $X$ of rank either $2n+1$ or $2n$, to be specified, and equip $V$ with a nondegenerate quadratic form taking values in the trivial line bundle.  We will compute a basic degeneracy class on $X$.

Let $E_p \subset V$ be an isotropic subbundle; when the rank of $V$ is odd, $E_p$ has rank $n+1-p$, and in the even rank case, $E_p$ has rank $n-p$.  Let $F'\subset V$ be an isotropic line bundle.  We assume $E_p \subseteq E$ and $F'\subseteq F$ are contained in some fixed maximal isotropic subbundles (of rank $n$).

The quadratic form induces isomorphisms $V\isom V^*$, and more generally $V/E^\perp \isom E^*$, for isotropic subbundles $E\subset V$.

We assume that the bundles $E_p$ and $F'$ are in general position, so the locus on $X$ where $\dim(E_p \cap F') \geq 1$ has codimension $p+n-1$.  This locus, $\Omega = \{x\in X\,|\, E_p \supseteq F'\}$, is the one whose class we will compute.

Let $\QQ(V) \xrightarrow{\pi} X$ be the quadric bundle associated to $V$, with tautological bundle $S=S_1\subset V$.  
The line bundle $F'$ defines a section
\[
  s\colon X \to \QQ(V),
\]
and the task is to compute $[\Omega]=s^*[\P(E_p)]$.

\begin{prop*}
When $V$ has odd rank, we have
\[
2 s^*[\P(E_p)] = c_{p+n-1}(V-E_p-F'-M),
\]
where $M=F^\perp/F$.  
When $V$ has even rank, we have
\[
 2 s^*[\P(E_p)] =  c_{p+n-1}(V-E_p-F') - e(E_p,F'),
\]
where $e(E_p,F') = (-1)^{\dim(E\cap F)}c_{p+n-1}(E/E_p + F/F')$.
\end{prop*}

\noindent
Recall that the parity of $\dim(E\cap F)$ is constant in (connected) families, so the sign is well defined.

The proof of the proposition relies on the presentation of the Chow ring of quadric bundles.

\begin{thm*}[{\cite[Theorem~7]{EG}, \cite[Theorem~B.1]{A}}]
With the notation as above, write $f=[\P(F)]$ in $A^*\QQ(V)$ and $h=c_1(S^*)$.  Then $A^*\QQ(V)=(A^*X)[h,f]/I$, where $I$ is generated by
\begin{align}
 2f &= h^n + c_1(V/F^\perp) h^{n-1} + \cdots + c_n(V/F^\perp),\label{e.rel-odd1}\\
 f^2  &= (-1)^{n}(c_{n}(F) + c_{n-2}(F)h^2 + \cdots) f
\end{align}
when $V$ has rank $2n+1$, and by
\begin{align}
 2hf &= h^n - c_1(F) h^{n-1} + \cdots + (-1)^n c_n(F), \label{e.rel-even1}\\
 f^2  &= (-1)^{n-1}(c_{n-1}(F) + c_{n-3}(F)h^2 + \cdots) f
\end{align}
when $V$ has rank $2n$.
\end{thm*}

We need a lemma:
\begin{lem*}
Let $k\geq 0$.  If the rank of $V$ is odd, then we have
\[
  2\pi_*(h^k ef)= c_{k+1}(V-E-F^\perp) = c_{k+1}(V-E-F-M).
\]
If the rank of $V$ is even, we have
\[
  2\,\pi_*(h^k ef) = \begin{cases} c_k(V-E-F) &\text{ if } k>0; \\ 1 - (-1)^{\dim(E\cap F)} & \text{ if } k=0. \end{cases}
\]
\end{lem*}

\begin{proof}
First suppose $V$ has rank $2n+1$.  Let $\rho\colon \P(E) \to X$ be the projection.  Using the relation \eqref{e.rel-odd1}, we get
\begin{align*}
  2\pi_*(h^k e f) &= \pi_*( e\cdot (h^{n+k} + c_1(V/F^\perp) h^{n+k-1} + \cdots + c_n(V/F^\perp) h^k ) ) \\
    &= \rho_*\left( h^{n+k} + c_1(V/F^\perp) h^{n+k-1} + \cdots + c_n(V/F^\perp) h^k \right) \\
    &= c_{k+1}(V-E-F^\perp),
\end{align*}
as claimed.

In the case $V$ has rank $2n$, we have
\[
  \rho_*(h^{k+n-1} + c_1(V/F) h^{k+n-2} + \cdots + c_n(V/F)h^{k-1}) = c_k(V-E-F)
\]
for $k>0$.  Using the relation \eqref{e.rel-even1}, this yields the desired formula, $2\pi_*(h^k ef) = c_k(V-E-F)$.    
The case $k=0$ is proved by taking $X$ to be a point and applying, e.g., \cite[Lemma 2]{EG}.
\end{proof}


\begin{proof}[Proof of Proposition]
We have $s(X) = \P(F')$, so that $s^*[\P(E_p)] = \pi_*( [\P(E_p)]\cdot [\P(F')] )$.

First consider the case when the rank of $V$ is odd.  Then $[\P(E_p)] = c_{p-1}(E/E_p \otimes S^*)\cdot e$ and $[\P(F')] = c_{n-1}(F/F'\otimes S^*)\cdot f$, so we must compute
\[
 2\,\pi_*( c_{p-1}(E/E_p \otimes S^*) \cdot c_{n-1}(F/F' \otimes S^*) \cdot ef ).
\]
Using the Lemma, this is equal to
\begin{align*}
 & 2\sum_{i=0}^{p-1} \sum_{j=0}^{n-1}  \pi_*\left( c_{p-1-i}(E/E_p) \cdot c_{n-1-j}(F/F') \cdot h^{i+j} e f \right) 
 \\ & \quad = 2\sum_{k=0}^{p+n-2}\left(\sum_{i+j=k} c_{p-1-i}(E/E_p)\cdot c_{n-1-j}(F/F') \right) \cdot \pi_*( h^k e f )  \\
    & \quad = 2\sum_{k=0}^{p+n-2} c_{p+n-2-k}(E/E_p + F/F') \cdot \pi_*( h^k e f )  \\
   & \quad = \sum_{k=0}^{p+n-2} c_{p+n-2-k}(E/E_p + F/F') \cdot  c_{k+1}(V-E-F-M)  \\
   &\quad = c_{p+n-1}(V-E_p-F'-M),
\end{align*}
where the last line uses $c_{p+n-1}(E/E_p + F/F') = 0$.

Next consider the case where $V$ has rank $2n$.  Now $[\P(E_p)] = c_p(E/E_p\otimes S^*)\cdot e$, and $[\P(F')] = c_{n-1}(F/F'\otimes S^*)\cdot f$.  An analogous computation gives
\begin{align*}
 & 2\,\pi_*( c_p(E/E_p\otimes S^*)\cdot c_{n-1}(F/F'\otimes S^*) \cdot e f ) \\
 & \quad = 2\sum_{k=0}^{p+n-1} c_{p+n-1-k}(E/E_p + F/F') \cdot \pi_*(h^k ef) \\
 & \quad = (1-(-1)^{\dim(E\cap F)})\,c_{p+n-1}(E/E_p+F/F') \\
  &\qquad + \sum_{k=1}^{p+n-1} c_{p+n-1-k}(E/E_p + F/F')  c_k(V-E-F) \\
 & \quad = c_{p+n-1}(V - E_p - F') - (-1)^{\dim(E\cap F)} c_{p+n-1}(E/E_p + F/F'). \tag*{\qedhere}
\end{align*}
\end{proof}


\end{document}